\documentclass[preprint,12pt]{elsarticle}

\usepackage{graphicx}

\usepackage{amssymb}

\usepackage{lineno}

\usepackage{amsfonts}
\usepackage{amsmath}
\usepackage{amsthm}
\usepackage{psfrag,rotating}
\newtheorem{theorem}{Theorem}[section]

\newtheorem{remark}{Remark}[section]

\newtheorem{corollary}{Corollary}[section]

\usepackage{bm,todonotes}

\usepackage{algorithm}
\usepackage{algorithmic}

\usepackage{color,hyperref}					
\hypersetup{colorlinks,breaklinks,
           linkcolor=blue,urlcolor=blue,
           anchorcolor=blue,citecolor=blue}

\usepackage[caption=false]{subfig}
\usepackage{float}
\usepackage{comment}
\graphicspath{{fig/}}

\usepackage{pstricks-add}

\newbox{\LegendeSC}
\savebox{\LegendeSC}{
    \begin{pspicture}(0,0)(0.6,0)
    \psdots[dotstyle=square,linecolor=blue,dotsize=5pt](0.25,0.1)
    \psline[linewidth=0.03,linecolor=blue](0,0.1)(0.5,0.1)
    \end{pspicture}}

\newbox{\LegendeCL}
\savebox{\LegendeCL}{
    \begin{pspicture}(0,0)(0.5,0)
    \psdots[dotstyle=o,linecolor=red,dotsize=5pt](0.25,0.1)
    \psline[linewidth=0.03,linecolor=red](0,0.1)(0.5,0.1)
    \end{pspicture}}

\newbox{\LegendeIW}
\savebox{\LegendeIW}{
    \begin{pspicture}(0,0)(0.5,0)
    \psdots[dotstyle=triangle,linecolor=black,dotsize=5pt,dotangle=180](0.25,0.1)
    \psline[linewidth=0.03,linecolor=black](0,0.1)(0.5,0.1)
    \end{pspicture}}

\newrgbcolor{darkgreen}{0 0.4 0}
\newbox{\LegendeW}
\savebox{\LegendeW}{
    \begin{pspicture}(0,0)(0.5,0)
    \psdots[dotstyle=triangle,linecolor=darkgreen,dotsize=5pt](0.25,0.1)
    \psline[linewidth=0.03,linecolor=darkgreen](0,0.1)(0.5,0.1)
    \end{pspicture}}
    
\newbox{\LegendeWtwo}
\savebox{\LegendeWtwo}{
    \begin{pspicture}(0,0)(0.5,0)
    \psdots[dotstyle=triangle,linecolor=darkgreen,dotsize=5pt](0.25,0.1)
    \psline[linestyle=dashed,dash=2pt 1pt 0.5pt 1pt,linewidth=0.03,linecolor=darkgreen](0,0.1)(0.5,0.1)
    \end{pspicture}}    
    
\newbox{\LegendeWthree}
\savebox{\LegendeWthree}{
    \begin{pspicture}(0,0)(0.5,0)
    \psdots[dotstyle=triangle,linecolor=darkgreen,dotsize=5pt](0.25,0.1)
    \psline[linestyle=dashed,dash=2pt 1pt,linewidth=0.03,linecolor=darkgreen](0,0.1)(0.5,0.1)
    \end{pspicture}}
    
\newbox{\LegendeIWtwo}
\savebox{\LegendeIWtwo}{
    \begin{pspicture}(0,0)(0.5,0)
    \psdots[dotstyle=triangle,linecolor=black,dotsize=5pt](0.25,0.1)
    \psline[linestyle=dashed,dash=2pt 1pt 0.5pt 1pt,linewidth=0.03,linecolor=black](0,0.1)(0.5,0.1)
    \end{pspicture}}    
    
\newbox{\LegendeIWthree}
\savebox{\LegendeIWthree}{
    \begin{pspicture}(0,0)(0.5,0)
    \psdots[dotstyle=triangle,linecolor=black,dotsize=5pt](0.25,0.1)
    \psline[linestyle=dashed,dash=2pt 1pt,linewidth=0.03,linecolor=black](0,0.1)(0.5,0.1)
    \end{pspicture}}

\journal{Elsevier}

\begin{document}

\begin{frontmatter}



\title{A weighted $\ell_1$-minimization approach for sparse polynomial chaos expansions}

\author[label1]{Ji Peng}
\author[label2]{Jerrad Hampton}
\author[label2]{Alireza Doostan\corref{cor1}}
\ead{alireza.doostan@colorado.edu}

\cortext[cor1]{Corresponding Author: Alireza Doostan}

\address[label1]{Mechanical Engineering Department, University of Colorado, Boulder, CO 80309, USA}
\address[label2]{Aerospace Engineering Sciences Department, University of Colorado, Boulder, CO 80309, USA}

%
\begin{abstract}
%

This work proposes a method for sparse polynomial chaos (PC) approximation of high-dimensional stochastic functions based on non-adapted random sampling. We modify the standard $\ell_1$-minimization algorithm, originally proposed in the context of compressive sampling, using {\it a priori} information about the decay of the PC coefficients and refer to the resulting algorithm as {\it weighted $\ell_1$-minimization}. We provide conditions under which we may guarantee recovery using this weighted scheme. Numerical tests are used to compare the weighted and non-weighted methods for the recovery of solutions to two differential equations with high-dimensional random inputs: a boundary value problem with a random elliptic operator and a $2$-D thermally driven cavity flow with random boundary condition.

\end{abstract}

\begin{keyword}
Compressive sampling \sep Sparse approximation \sep Polynomial chaos \sep Basis pursuit denoising (BPDN) \sep Weighted $\ell_1$-minimization \sep Uncertainty quantification  \sep Stochastic PDEs
\end{keyword}
\end{frontmatter}


%
\section{Introduction}
\label{sec:intro}
%

As we analyze engineering systems of increasing complexity, we must strategically confront the imperfect knowledge of the underlying physical models and their inputs, as well as the implied imperfect knowledge of a quantity of interest (QOI) predicted from these models. The understanding of outputs as a function of inputs in the presence of such uncertainty falls within the field of uncertainty quantification. The accurate quantification of the uncertainty of the QOI allows for the rigorous mitigation of both unfounded confidence and unnecessary diffidence in the anticipated QOI. 

Probability is a natural mathematical framework for describing uncertainty, and so we assume that the system input is described by a vector of independent random variables, $\bm{\Xi}$. If the random variable QOI, denoted by $u(\bm{\Xi})$, has finite variance, then the polynomial chaos (PC) expansion~\cite{Ghanem02, Xiu02} is given in terms of the orthonormal polynomials $\{\psi_j(\bm{\Xi})\}$ as
\begin{align}
\label{eq:PCE}
u(\bm{\Xi})=\sum_{j=1}^\infty c_j\psi_j(\bm{\Xi}).
\end{align}
A more detailed exposition on the use of PC expansion in this work is given in Section~\ref{subsubsec:PCE}.

To identify the PC coefficients, $c_j$ in~(\ref{eq:PCE}), sampling methods including Monte Carlo simulation~\cite{Reagan03}, pseudo-spectral stochastic collocation~\cite{Mathelin03,Xiu05a,Babuska07a,Constantine12a}, or least-squares regression~\cite{Hosder06} may be applied. These methods for evaluating the PC coefficients are popular in that deterministic solvers for the QOI may be used without being adapted to the probability space. However, the standard Monte Carlo approach suffers from a slow convergence rate. Additionally, a major limitation to the use of the last two approaches above is that the number of samples needed to approximate $c_j$ increases exponentially with the dimension of the input uncertainty, i.e., the number of random variables needed to describe the input uncertainty, see, e.g., \cite{LeMaitre10,Xiu10a,Doostan07b,Doostan09,Doostan2013a}. In this work, we use the Monte Carlo sampling method while considerably improving the accuracy of approximated PC coefficients (for the same number of samples) by exploiting the approximate sparsity of the coefficients $c_j$. As $u$ has finite variance, the $c_j$ in (\ref{eq:PCE}) necessarily converge to zero, and if this convergence is sufficiently rapid, then $u(\bm{\Xi})$ may be approximated by
\begin{align}
\label{eq:PCEt}
\hat{u}(\bm{\Xi})=\sum_{j\in\mathcal{C}}c_j\psi_j(\bm{\Xi}),
\end{align}
where the index set $\mathcal{C}$ has few elements. 
When this occurs we say that $\hat{u}$ is reconstructed from a sparse PC expansion, and that $u$ admits an approximately sparse PC representation. By truncating the PC basis implied by (\ref{eq:PCE}) to $P$ elements, we may perform calculations on the truncated PC basis.  
If we let $\bm{c}$ be a vector of $c_j$, for $j=1,\dots,P$, then the approximate sparsity of the QOI (implied by the sparsity of $\bm c$) and the practical advantage of representing the QOI with a small number of basis functions motivate a search for an approximate $\bm{c}$ 
which has few non-zero entries~\cite{Doostan10b,Blatman10,Doostan11a,Blatman11,Mathelin12a,Yan12,Yang13}. We seek to achieve an accurate reconstruction with a small number of samples, and so look to techniques from the field of compressive sampling~\cite{Chen98, Chen01,Donoho06b, Candes06a, Candes06b, Candes06c, Candes07a,Candes08b,Bruckstein09}.

Let $\bm{\xi}$ represent a realization of $\bm{\Xi}$. We define $\bm{\Psi}$ as the matrix where each row corresponds to the row vector of $P$ PC basis functions evaluated at sampled $\bm{\xi}$ with the corresponding $u(\bm{\xi})$ being an entry in the vector $\bm{u}$. We assume $N<P$ samples of $\bm{\xi}$, so that $\bm{\Psi}$ is $N\times P$, $\bm{c}$ is $P\times 1$, and $\bm{u}$ is $N\times 1$. Compressive sampling seeks a solution $\bm c$ with minimum number of non-zero entries by solving the optimization problem
\begin{equation}
\label{eq:CS_l0}
\mathcal{P}_{0,\epsilon}\equiv\{\mathop{\arg\min}\limits_{\bm{c}}\Vert\bm{c}\Vert_0 : \Vert\bm{\Psi c}-\bm{u}\Vert_2\leqslant\epsilon\}.
\end{equation}
Here $\Vert\bm{c}\Vert_0$ is defined as the number of non-zero entries of $\bm c$, and a solution to $\mathcal{P}_{0,\epsilon}$ directly provides an optimally sparse approximation in that a minimal number of non-zero entries are used to recover $\bm{u}$ to within $\epsilon$ in the $\ell_2$ norm. In general, the cost of finding a solution to $\mathcal{P}_{0,\epsilon}$ grows exponentially in $P$ \cite{Bruckstein09}. To resolve this exponential dependence, the convex relaxation of $\mathcal{P}_{0,\epsilon}$ based on $\ell_1$-minimization, also referred to as basis pursuit denoising (BPDN), has been proposed \cite{Chen98,Chen01,Candes06a,Donoho06b,Bruckstein09}. Specifically, BPDN seeks to identify $\bm{c}$ by solving
\begin{equation}
\label{eq:CS}
\mathcal{P}_{1,\epsilon}\equiv\{\mathop{\arg\min}\limits_{\bm{c}}\Vert\bm{c}\Vert_1 : \Vert\bm{\Psi c}-\bm{u}\Vert_2\leqslant\epsilon\}
\end{equation}
using convex optimization algorithms \cite{Chen98,Osborne00,Daubechies04,Combettes05,Figueiredo07,Kim07a,spgl1:2007,Donoho09a}. In practice, $\mathcal{P}_{0,\epsilon}$ and $\mathcal{P}_{1,\epsilon}$ may have similar solutions, and the comparison of the two problems has received significant study, see, e.g.,~\cite{Bruckstein09} and the references therein. 

Note in (\ref{eq:CS}) the constraint $\Vert\bm{\Psi c}-\bm{u}\Vert_2\leqslant\epsilon$ depends on the observed $\bm{\xi}$ and $u(\bm{\xi})$; not in general $\bm{\Xi}$ and $u(\bm{\Xi})$. As a result, $\bm{c}$ may be chosen to fit the input data, and not accurately approximate $u(\bm{\Xi})$ for previously unobserved realizations $\bm{\xi}$. To avoid this situation, we determine $\epsilon$ by cross-validation~\cite{Doostan11a} as discussed in Section~\ref{sec:cross_validation}.

To assist in identifying a solution to (\ref{eq:CS}), note that for certain classes of functions, theoretical analysis suggests estimates on the decay for the magnitude of the PC coefficients~\cite{Babuska04,Cohen10a,Beck12a}. Alternatively, as we shall see in Section \ref{sec:L1_example_cavity}, such estimates may be derived by taking into account certain relations among physical variables in a problem. It is reasonable to use this {\it a priori} information to improve the accuracy of sparse approximations~\cite{Escoda06}. Moreover, even if this decay information is unavailable, each approximated set of PC coefficients may be considered as an initialization for the calculation of an improved approximation, suggesting an iterative scheme~\cite{Candes08a, Escoda06,Chartrand08,Khajehnejad10,Mathelin12a,Yang13}.

In this work, we explore the use of {\it a priori} knowledge of the PC coefficients as a weighting of $\ell_1$ norm in BPDN in what is referred to as {\it weighted $\ell_1$-minimization} (or {\it weighted BPDN}),
\begin{equation}
\label{eq:CS_w}
\mathcal{P}_{1,\epsilon}^{(\bm W)}\equiv\{\mathop{\arg\min}\limits_{\bm{c}}\Vert\bm{Wc}\Vert_1 : \Vert\bm{\Psi c}-\bm{u}\Vert_2\leqslant\epsilon\},
\end{equation}
where $\bm W$ is a diagonal matrix to be specified. Previously, $\ell_1$-minimization has been applied to solutions of stochastic partial differential equations with approximately sparse $\bm{c}$~\cite{Doostan10b,Doostan11a,Mathelin12a,Yang13}, but these approximately sparse $\bm{c}$ include a number of small magnitude entries which inhibit the accurate recovery of larger magnitude entries. The primary goal of this work is to utilize {\it a priori} information about $\bm{c}$, in the form of estimates on the decay of its entries, to reduce this inhibition and enhance the recovery of a larger proportion of PC coefficients; in particular those of the largest magnitude. We provide theoretical results pertaining to the quality of the solution identified from the weighted $\ell_1$-minimization problem (\ref{eq:CS_w}). 

The rest of this paper is structured as follows. In Section \ref{sec:prob_set}, we introduce the problem of interest as well as our approach for the stochastic expansion of its solution. Following that, in Section \ref{sec:l1}, we present our results on weighted $\ell_1$-minimization and its corresponding analysis for sparse PC expansions. In Section \ref{sec:examples}, we provide two test cases which we use to describe the specification of the weighted $\ell_1$-minimization problem and explore its performance and accuracy. In particular, in Section \ref{sec:L1_example_cavity}, we utilize a simple dimensional relation to derive approximate upper bounds on the PC expansion coefficients of the velocity field in a flow problem.

%
\section{Problem Statement and Solution Approach}
\label{sec:prob_set}
%

%
\subsection{PDE formulation}
\label{subsubsec:PDE}

Let the random vector $\bm{\Xi}$, defined on the probability space $(\Omega,\mathcal{F},\mathbb{P})$, characterize the input uncertainties and consider the solution of a partial differential equation defined on a bounded Lipschitz continuous domain $\mathcal{D}\subset\mathbb{R}^{D}$, $D\in\{1,~2,~3\}$, with boundary $\partial\mathcal{D}$. The uncertainty implied by $\bm{\Xi}$ may be represented in one or many relevant parameters, e.g., the diffusion coefficient, boundary conditions, and/or initial conditions. Letting $\mathcal{L},\mathcal{I},$ and $\mathcal{B}$ depend on the physics of the problem being solved, the solution $u$ satisfies the three constraints
\begin{equation}
\label{eqn:PDE_operator}
\begin{aligned}
&\mathcal{L}(\bm{x},t,\bm{\Xi};u(t,\bm{x},\bm{\Xi}))=0,~~& &\bm{x}\in\mathcal{D}, \\
&\mathcal{I}(\bm{x},\bm{\Xi};u(0,\bm{x},\bm{\Xi}))=0,~~& &\bm{x}\in\mathcal{D}, \\
&\mathcal{B}(\bm{x},t,\bm{\Xi};u(t,\bm{x},\bm{\Xi}))=0,~~& &\bm{x}\in\partial\mathcal{D}.
\end{aligned}
\end{equation}

We assume that $(\Omega,\mathcal{F},\mathbb{P})$ is formed by the product of $d$ probability spaces, $(\mathbb{R},\mathbb{B}(\mathbb{R}),\mathbb{P}_k)$ corresponding to each coordinate of $\bm{\Xi}$, denoted by $\Xi_k$; here $\mathbb{B}(\cdot)$ represents the Borel $\sigma$-algebra. We further assume that the random variable $\Xi_k$ is continuous and distributed according to the density $\rho_k$ implied by $\mathbb{P}_k$. Note that this entails $\Omega=\mathbb{R}^d$, $\mathcal{F}=\mathbb{B}(\mathbb{R}^d)$, that each $\Xi_k$ is independently distributed, and that the joint distribution for $\bm{\Xi}$, denoted by $\rho$, equals the tensor product of the marginal distributions $\{\rho_k\}$.

In this work, we assume that conditioned on the $i$th random realization of $\bm{\Xi}$, denoted by $\bm{\xi}^{(i)}$, the numerical solution to (\ref{eqn:PDE_operator}) may be calculated by a fixed solver; for our examples we use the finite element solver package FEniCS~\cite{LoggMardalEtAl2012a}. For any fixed $\bm{x}_0,t_0$, our objective is to reconstruct the solution $u(\bm{x}_0,t_0,\bm{\Xi})$ using $N$ realizations $\{u(\bm{x}_0,t_0,\bm{\xi}^{(i)})\}$. For brevity we suppress the dependence of $u(\bm{x}_0,t_0,\bm{\Xi})$ and $\{u(\bm{x}_0,t_0,\bm{\xi}^{(i)})\}$ on $\bm{x}_0$ and $t_0$.

The two specific physical problems we consider are a boundary value problem with a random elliptic operator and a $2$-D heat driven cavity flow with a random boundary condition. 

\subsection{Polynomial Chaos (PC) expansion}
\label{subsubsec:PCE}

Our methods to approximate the solution $u$ to (\ref{eqn:PDE_operator}) make use of the PC basis functions which are induced by the probability space $(\Omega,\mathcal{F},\mathbb{P})$ on which $\bm{\Xi}$ is defined. Specifically, for each $\rho_k$ we define $\{\psi_{k,j}\}_{j\ge 0}$ to be the complete set of orthonormal polynomials of degree $j$ with respect to the weight function $\rho_k$ \cite{Askey85,Xiu02}. As a result, the orthonormal polynomials for $\bm{\Xi}$ are given by the products of the univariate orthonormal polynomials,
\begin{equation}
\label{eqn:PCProdDef}
\psi_{\bm{\alpha}}(\bm{\Xi})=\mathop{\prod}\limits_{k=1}^d\psi_{k,\alpha_k}(\Xi_k),
\end{equation}
where each $\alpha_k$, representing the $k$th coordinate of the multi-index $\bm{\alpha}$, is a non-negative integer. For computation, we truncate the expansion in (\ref{eq:PCE}) to the set of $P$ basis functions associated with the subspace of polynomials of total order not greater than $q$, that is $\sum_{k=1}^d\alpha_k\le q$. For convenience, we also order these $P$ basis functions so that they are indexed by $\{1,\cdots,P\}$ as opposed to the vectorized indexing in (\ref {eqn:PCProdDef}). The basis set $\{\psi_{j}\}_{j=1}^P$ has the cardinality
\begin{equation}
\label{eqn:P}
P = \frac{(d+q)!}{d!q!}.
\end{equation}
For the interest of presentation, we interchangeably use both notations for representing PC basis. For any fixed $\bm{x}_0,t_0$, the PC expansion of $u$ and its truncation are then defined by
\begin{equation}
\label{eqn:PCE}
u(\bm{x}_0,t_0,\bm{\Xi}) = u(\bm{\Xi}) = \sum_{j=1}^{\infty}c_j\psi_j(\bm{\Xi})\approx \sum_{j=1}^{P}c_j\psi_j(\bm{\Xi}).
\end{equation}
Tough $u$ is an arbitrary function in $L_2(\Omega,\mathbb{P})$, we are limited to an approximation in the span of our basis polynomials, and the error incurred from this approximation is referred as {\it truncation error}.

In this work we assume that, for each $k$, $\rho_k$ is known {\it a priori}.  Two commonly used probability densities for $\rho_k$ are uniform and Gaussian; the corresponding polynomial bases are, respectively, Legendre and Hermite polynomials~\cite{Xiu02}. We furthermore set $\Xi_k$ to be uniformly distributed on $[-1,1]$ and our PC basis functions are constructed from the orthonormal Legendre polynomials. The presented methods, however, may be applied to any set of orthonormal polynomials and their associated random variables.

We use the samples $\bm{\xi}^{(i)}$, $i=1,\dots,N$,  of $\bm{\Xi}$ to evaluate the PC basis and identify a corresponding solution $u(\bm{\xi}^{(i)})$ to (\ref{eqn:PDE_operator}). This evaluated PC basis forms a row of $\bm{\Psi}\in \mathbb{R}^{N\times P}$ in (\ref{eq:CS}), that is $\bm{\Psi}(i,j)=\psi_{j}(\bm{\xi}^{(i)})$. The corresponding solution $u(\bm{\xi}^{(i)})$ is the associated element of the vector $\bm{u}$. We are then faced with identifying the vector of PC coefficients $\bm c\in\mathbb{R}^P$ in (\ref{eqn:PCE}), which we address by considering techniques from compressive sampling.

\subsection{Sparse PC expansion}
\label{subsec:CS}

As the PC expansion in (\ref{eqn:PCE}) is a sum of orthonormal random variables defined by $\psi_j(\bm{\Xi})$, the exact PC coefficients may be computed by projecting $u(\bm{\Xi})$ onto the basis functions $\psi_j(\bm{\Xi})$ such that
\begin{equation*}
c_j=\mathbb{E}\left[u(\bm{\Xi})\psi_j(\bm{\Xi})\right]=\int_\Omega u(\bm{\xi})\psi_j(\bm{\xi})\rho(\bm{\xi})d\bm{\xi}.
\end{equation*}
To compute the PC coefficients non-intrusively, besides the standard Monte Carlo sampling, which is known to converge slowly, we may estimate this expectation via, for instance, sparse grid quadrature. While this latter approach performs well when $d$ and $q$ are small, it may become impractical for high-dimensional random inputs. Alternatively, $\bm c$ may be computed from a discrete projection, e.g., least-squares regression \cite{Hosder06}, which generally requires $N>P$ solution realizations to achieve a stable approximation. 

We assume that $\bm{c}$ is approximately sparse, and seek to identify an appropriate $\mathcal{C}$, as in (\ref{eq:PCEt}), having a small number of elements and giving a small truncation error. To this end we extend ideas from the field of compressive sampling. If the number of elements of $\mathcal{C}$, denoted by $|\mathcal{C}|$, is small, then using only the columns in $\bm{\Psi}$ corresponding to elements of $\mathcal{C}$ reduces the dimension of the PC basis from $P$ to $|\mathcal{C}|$. This significantly reduces the number of PC coefficients requiring estimation and consequently the number of solution realizations $N$. We define $\bm{\Psi}_{\mathcal{C}}$ as the truncation of $\bm{\Psi}$ to those columns only relevant to the basis functions of $\mathcal{C}$, and similarly define $\bm{c}_{\mathcal{C}}$ as the truncation of $\bm{c}$. If $|\mathcal{C}|<N$, then the determination of $|\mathcal{C}|$ coefficients gives an optimization problem less prone to overfit the data \cite{Hansen98b}, even when $N<P$. For example, the least-squares approximation of $\bm{c}_\mathcal{C}$, $\hat{\bm{c}}_\mathcal{C}=(\bm{\Psi}_\mathcal{C}^{T}\bm{\Psi}_\mathcal{C})^{-1}\bm{\Psi}_\mathcal{C}^{T}\bm{u}$, minimizing $\|\bm{\Psi}_\mathcal{C}\bm{c}_\mathcal{C}-\bm{u}\|_2$ is well-posed and will have a unique solution if $\bm{\Psi}_\mathcal{C}$ is of full rank.

Note that the identification of $\mathcal{C}$ is critical to the optimization problem $\mathcal{P}_{0,\epsilon}$ in (\ref{eq:CS}). If we instead have a solution to $\mathcal{P}_{1,\epsilon}$, then we may infer a $\mathcal{C}$ by noting the entries of the approximated $\bm{c}$ which have magnitudes above a certain threshold. 
{\color{black}Motivated to obtain more accurate sparse solutions, we next introduce a compressive sampling technique which modifies $\mathcal{P}_{1,\epsilon}$ by weighting each $c_j$ differently in $\Vert \bm c\Vert_{1}$. As we shall discuss later, these weights are generated based on some {\it a priori} information on the decay of $c_j$, when available.} 



%
\section{\texorpdfstring{Weighted $\ell_1$-minimization}{BPDN and weighted BPDN}}
\label{sec:l1}
%


To develop a weighted $\ell_1$-minimization $\mathcal{P}^{(\bm{W})}_{1,\epsilon}$, we do not consider any changes to the algorithm solving $\mathcal{P}_{1,\epsilon}$, but instead transform the problem with the use of weights, such that the same solver may be used. We define the diagonal weight matrix $\bm{W}$, with diagonal entries $w_j\ge 0$, and consider the new weighted problem $\mathcal{P}^{(\bm{W})}_{1,\epsilon}$ in (\ref{eq:CS_w}) with 
\begin{equation}
\label{eqn:weighted_l1_norm}
\Vert\bm{W}\bm{c}\Vert_1=\sum_{j=1}^{P}w_j\vert c_j\vert.
\end{equation}
%
%

If {\it a priori} information is available for $c_j$, it is natural to use it to define $\bm{W}$~\cite{Escoda06}. Heuristically, columns with large anticipated $|c_j|$ should not be heavily penalized when used in the approximation, that is the corresponding $w_j$ should be small. In contrast, $|c_j|$ which are not expected to be large should be paired with large $w_j$. This suggests allowing $w_j$ to be inversely related to $|c_j|$, \cite{Candes08a},
\begin{equation}
\label{eqn:weight_norm1}
w_j=\left\{\begin{aligned}
&\vert c_{j}\vert^{-p},& c_j\neq0,\\
&\infty ,& c_j=0.
\end{aligned}
\right.
\end{equation}
The parameter $p\in [0,1]$ may be used to account for the confidence in the anticipated $|c_j|$. Large values of $p$ lead to more widely dispersed weights and indicate greater confidence in these $|c_j|$ while small values lead to more clustered weights and indicate less confidence in these $|c_j|$. These weights deform the $\ell_1$ ball, as Fig. \ref{fig:schematic} shows, to discourage small coefficients from the solution and consequently enhance the accuracy. A detailed discussion of weighted $\ell_1$-minimization and examples in signal processing are given in~\cite{Candes08a}.


%
\begin{figure}[htb] 
    \centering
    \begin{tabular}{cc}
    \includegraphics[trim = 0mm 5mm 0mm 10mm, clip, width=0.35\textwidth]{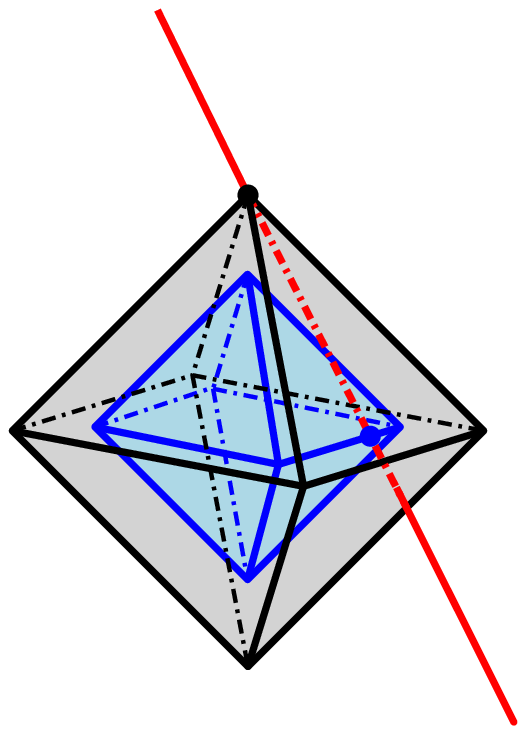}  
    \put(-10,-10){\footnotesize $\bm{\Psi c} = \bm{\Psi}\bm{c}_0$}
    \put(-67,130){\footnotesize $\bm{c}_0$}
    \put(-50,64){\footnotesize $\bm{c}$}    
    &
    \hspace{2cm}    
    \includegraphics[trim = 0mm 5mm 0mm 12mm, clip, width=0.26\textwidth]{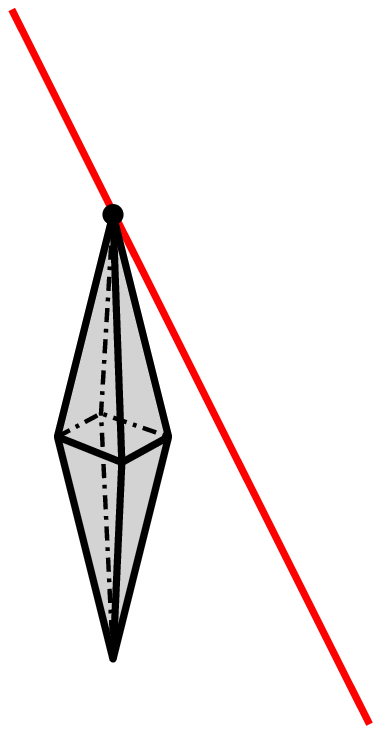}
    \put(-10,-10){\footnotesize $\bm{\Psi c} = \bm{\Psi}\bm{c}_0$}
    \put(-67,128){\footnotesize $\bm{c}_0$, $\bm{c}$}
    \\
    (a) & \hspace{1.5cm} (b)
      
     \end{tabular}
      \caption{Schematic of approximation of a sparse $\bm c_0\in\mathbb{R}^{3}$ via standard and weighted $\ell_1$-minimization (based on \cite{Candes08a}). (a) Standard $\ell_1$-minimization where, depending on $\bm\Psi$, the problem $\mathcal{P}_{1,0}$ with $\bm{u} = \bm{\Psi}\bm{c}_0$ may have a solution $\bm{c}$ such that $\Vert\bm{c}\Vert_1\le \Vert\bm{c}_0\Vert_1$. (b) Weighted $\ell_1$-minimization for which there is no $\bm{c}$ with $\Vert\bm{Wc}\Vert_1\le \Vert\bm{W}\bm{c}_0\Vert_1$.}            
\label{fig:schematic}       
\end{figure}

As in~\cite{Chartrand08, Candes08a}, to insure stability, we consider a damped version of $w_j$ in (\ref{eqn:weight_norm1}), 
\begin{equation}
\label{eqn:weight}
w_j = \left(\vert c_j\vert + \epsilon_w\right)^{-p},
\end{equation}
where $\epsilon_w$ is a relatively small positive parameter. In the numerical examples of this paper, we set $\epsilon_w=5\times 10^{-5}\cdot \hat{c}_1$ to generate $w_j$ in $\mathcal{P}^{(\bm{W})}_{1,\epsilon}$, where $\hat{c}_1 = \frac{1}{N}\sum_{i=1}^{N}u(\bm{\xi}^{(i)})$ is the Monte Carlo estimate of the degree zero PC coefficient (or, equivalently, the sample average of $u$).  

\begin{remark}[Choice of $p$ in (\ref{eqn:weight})] When defined based on the exact values $\vert c_j\vert$, the weights $w_j$ in (\ref{eqn:weight}) together with (\ref{eqn:weighted_l1_norm}) imply an $\ell_r$-minimization problem of the form $\mathcal{P}_{r,\epsilon}\equiv\{\mathop{\arg\min}\limits_{\bm{c}}\Vert\bm{c}\Vert_r : \Vert\bm{\Psi c}-\bm{u}\Vert_2\leqslant\epsilon\}$ to solve for $\bm{c}$, where $r=1-p\in[0,1]$. Depending on the value of $r$, such a minimization problem may outperform the standard $\ell_1$-minimization, see, e.g., \cite{Chartrand08}. In practice, however, an optimal selection of $r$ (or $p$) is not a trivial task and necessitates further analysis. In the present study, similar to \cite{Candes08a}, we choose $p=1$. 

\end{remark}

\subsection{Setting weights $w_j$}
\label{sec:setting_weights}

As the true $\bm{c}$ is unknown, an approximation of $\bm{c}$ must be employed to form the weights. In~\cite{Chartrand08,Candes08a,Needell09,Yang13} an iterative approach is proposed wherein these weights are computed from the previous approximation of $\bm{c}$. More precisely, at iteration $l+1$, the weights are set by 
\begin{equation*}
w_j = \left(\vert \hat c_j^{(l)}\vert + \epsilon_w\right)^{-1},
\end{equation*}
where $\hat c_j^{(l)}$ is the estimate of $c_j$ obtained from $\mathcal{P}^{(\bm{W})}_{1,\epsilon}$ at iteration $l$ and $w_j=1$ at iteration $l=1$. However, the solution to such iteratively re-weighted $\ell_1$-minimization problems may be expensive due to the need for multiple $\mathcal{P}^{(\bm{W})}_{1,\epsilon}$ solves. Additionally, the convergence of the iterates is not always guaranteed \cite{Candes08a}. Moreover, as we will observe from the results of Section \ref{sec:examples}, unlike the weighted $\ell_1$-minimization, the accuracies obtained from the iteratively re-weighted $\ell_1$-minimization approach are sensitive to the choice of $\epsilon_w$. In particular, for relatively large or small values of $\epsilon_w$, the iteratively re-weighted  $\ell_1$-minimization may even lead to less accurate results as compared to the standard $\ell_1$-minimization.

Alternatively, to set $w_j$, we here focus our attention on situations where {\it a priori} knowledge on $c_j$ in the form of decay of $\vert c_j\vert$ are available. This includes primarily a class of linear elliptic PDEs with random inputs \cite{Babuska04,Cohen10a,Beck12a}. We also provide preliminary results on a non-linear problem, specifically a $2$-D Navier-Stokes equation, for which we exploit a simple physical dependency among solution variables to generate the approximate decay of $\vert c_j\vert$. We notice that the success of our weighted $\ell_1$-minimization depends on the ability of our approximate $\vert c_j \vert$ to reveal {\it relative importance} of $\vert c_j\vert$ rather than their precise values. As we shall empirically illustrate in Section \ref{sec:examples}, when such decay information is used, the weighted $\ell_1$-minimization approach outperforms the iteratively re-weighted $\ell_1$-minimization. 

To solve $\mathcal{P}^{(\bm{W})}_{1,\epsilon}$, the standard $\ell_1$-minimization solvers may be used. In this work we use the MATLAB package SPGL1~\cite{spgl1:2007} based on the spectral projected gradient algorithm~\cite{Berg08}. Specifically, $\tilde{\bm{c}} = \bm{W}\bm{c}$ may be solved from $\mathcal{P}_{1,\epsilon}$ with the modified measurement matrix $\tilde{\bm{\Psi}} = \bm{\Psi}\bm{W}^{-1}$. We then set $\bm{c} = \bm{W}^{-1}\tilde{\bm{c}}$. 

We defer presenting examples of setting $w_j$ to Section \ref{sec:examples} and instead provide theoretical analysis on the quality of the solution to the weighted $\ell_1$-minimization problem $\mathcal{P}^{(\bm{W})}_{1,\epsilon}$. In particular, we limit our theoretical analysis to determining if  $\mathcal{P}^{(\bm{W})}_{1,\epsilon}$ is equivalent to solving $\mathcal{P}_{0,\epsilon}$, finding an optimally sparse solution $\bm{c}$.

\subsection{Theoretical recovery via weighted $\ell_1$-minimization}
\label{subsec:wl1_theory}

Following the ideas of \cite{Candes06d,Juditsky11a,Juditsky11b,Gribonval03,Donoho06b,Cohen09a}, we consider analysis which depends on vectors in the kernel of $\bm{\Psi}$. We consider $\bm{c}_0$ to be a sparse approximation, such that $\bm{\Psi}\bm{c}_0 +\bm{e}= \bm{u}$ where $\|\bm{e}\|_2\le \epsilon$ indicates a small level of truncation error and/or noise is present, implying that exact reconstructions are themselves approximated by a sparse solution. Stated another way, $\bm{c}_0$ is a solution to $\mathcal{P}_{0,\epsilon}$. Let $\bm{c}_1$ be a solution to $\mathcal{P}^{(\bm{W})}_{1,\epsilon}$. Further, let $\mathcal{C}=\mbox{Supp}(\bm{c}_0)$, and note that $s = |\mathcal{C}|$ is the sparsity of $\bm{c}_0$. 





The following theorem is closely related to Theorem 1 of \cite{Candes06d} and provides a condition to compare a solution to $\mathcal{P}^{(\bm{W})}_{1,\epsilon}$ with a solution to $\mathcal{P}_{0,\epsilon}$, in terms of the Restricted Isometry Constant (RIC) $\delta_s$, \cite{Candes05a,Candes06d}; defined such that for any vector, $\bm{x}\in\mathbb{R}^P$, supported on at most $s$ entries,
\begin{align}
\label{eqn:RIC}
(1-\delta_{s})\|\bm{x}\|^2_2\le \|\bm{\Psi}\bm{x}\|^2_2 \le(1+\delta_{s})\|\bm{x}\|^2_2.
\end{align}
While we follow Theorem 1 of \cite{Candes06d} due to the simplicity of its proof, we note that improved conditions on the RIC have been presented in more recent studies~\cite{Mo11,Andersson12}. 

\begin{theorem}
Let $s$ be such that $\delta_{3s}+3\delta_{4s}<2$. Then for any approximate solution, $\bm{c}_0$, supported on $\mathcal{C}$ with $|\mathcal{C}|\le s$, any solution $\bm{c}_1$ to $\mathcal{P}^{(\bm{W})}_{1,\epsilon}$ obeys
\begin{align*}
\|\bm{c}_0-\bm{c}_1\|_2\le C\cdot \epsilon, \end{align*}
where the constant $C$ depends on $s$, $\max_{j\in\mathcal{C}}w_j$, and $\min_{j\in\mathcal{C}^c}w_j$.
\end{theorem}
\begin{proof}
Our proof is essentially an extension of the proof of Theorem 1 in \cite{Candes06d} to account for the weighted $\ell_1$ norm. Let $\bm{h}:=\bm{c}_1-\bm{c}_0$. Note that as $\bm{c}_1 = \bm c_0 + \bm h$ solves the weighted $\ell_1$-minimization problem $\mathcal{P}^{(\bm{W})}_{1,\epsilon}$,
\begin{align*}
\|\bm{W}\bm{c}_0\|_1-\|\bm{W}\bm{h}\|_{\mathcal{C},1}+\|\bm{W}\bm{h}\|_{\mathcal{C}^c,1}\le \|\bm{W}\left(\bm{c}_0+\bm{h}\right)\|_1= \|\bm{W}\bm{c}_1\|_1\le \|\bm{W}\bm{c}_0\|_1,
\end{align*}
where we use notation for an $\ell_r$ norm restricted to coordinates in a set $\mathcal{S}$ as $\|\bm{x}\|_{\mathcal{S},r}$. It follows that for some $0\le \beta\le 1$,
\begin{align}
\label{eqn:betta_def}
\|\bm{W}\bm{h}\|_{\mathcal{C}^c,1}\le \beta\|\bm{W}\bm{h}\|_{\mathcal{C},1}.
\end{align}
Sort the entries of $\bm h$ supported on $\mathcal{C}^c$ in descending order of their magnitudes, divide $\mathcal{C}^c$ into subsets of size $M$, and enumerate these sets as $\mathcal{C}_1,\cdots,\mathcal{C}_{n}$, where $\mathcal{C}_1$ corresponds to the indices of the $M$ largest entries of sorted $\bm h$, $\mathcal{C}_2$ corresponds to the indices of the next $M$ largest entries of sorted $\bm h$, and so on. Let $\mathcal{S}=\mathcal{C}\cup\mathcal{C}_1$, and note that the $k$th largest (in magnitude) entry of any $\bm x$ accounts for less than $1/k$ of the $\Vert \bm x\Vert_1$, so that
\begin{align*}
\|\bm{h}\|_{\mathcal{S}^c,2}^2 = \sum_{k\in \mathcal{S}^c} h_k^2\le\|\bm{h}\|^2_{\mathcal{C}^c,1}\mathop{\sum}\limits_{k=M+1}^{P}k^{-2}\le\|\bm{h}\|_{\mathcal{C}^c,1}^2\cdot\frac{1}{M}.
\end{align*}
We now bound the unweighted $\ell_1$ norm from above by the weighted $\ell_1$ norm, to achieve
\begin{align*}
\|\bm{h}\|_{\mathcal{C}^c,1}^2\cdot\frac{1}{M}\le \|\bm{Wh}\|_{\mathcal{C}^c,1}^2\cdot\frac{1}{M\min_{i\in\mathcal{C}^c}w_i^2}.
\end{align*}
From the condition (\ref{eqn:betta_def}),
\begin{align*}
\|\bm{Wh}\|_{\mathcal{C}^c,1}^2\cdot\frac{1}{M\min_{i\in\mathcal{C}^c}w_i^2}\le\|\bm{Wh}\|_{\mathcal{C},1}^2\cdot\frac{\beta^2}{M\min_{i\in\mathcal{C}^c}w_i^2}.
\end{align*}
Bounding the weighted $\ell_1$ norm from above by the unweighted $\ell_1$ norm gives,
\begin{align*}
\|\bm{Wh}\|_{\mathcal{C},1}^2\cdot\frac{\beta^2}{M\min_{i\in\mathcal{C}^c}w_i^2}\le \|\bm{h}\|_{\mathcal{C},1}^2\cdot\frac{\beta^2\max_{j\in\mathcal{C}}w_j^2}{M\min_{i\in\mathcal{C}^c}w_i^2}.
\end{align*}
Bounding this by the $\ell_2$ norm yields the desired inequality,
\begin{align*}
\|\bm{h}\|_{\mathcal{S}^c,2}^2\le\|\bm{h}\|_{\mathcal{C},2}^2\cdot\frac{\beta^2|\mathcal{S}|\max_{j\in\mathcal{C}}w_j^2}{M\min_{i\in\mathcal{C}^c}w_i^2}.
\end{align*}
Let 
\begin{align*}
\eta:=\frac{\beta^2|\mathcal{S}|\max_{j\in\mathcal{C}}w_j^2}{M\min_{i\in\mathcal{C}^c}w_i^2}.
\end{align*}
It follows that
\begin{align*}
\|\bm{h}\|_2^2 =  \|\bm{h}\|_{\mathcal{S}^c,2}^2 + \|\bm{h}\|_{\mathcal{S},2}^2\le (1+\eta)\|\bm{h}\|_{\mathcal{C},2}^2.
\end{align*}
Following the proof from Theorem 1 of \cite{Candes06d} we have that
\begin{align*}
\|\bm{\Psi}\bm{h}\|_2\ge \left(\sqrt{1-\delta_{M+|\mathcal{C}|}}-\frac{|\mathcal{C}|}{M}\sqrt{1+\delta_M}\right)\|\bm{h}\|_{\mathcal{C},2},
\end{align*}
and it follows that
\begin{align*}
\|\bm{h}\|_2&\le\sqrt{1+\eta}\|\bm{h}\|_{\mathcal{C},2}\le\frac{\sqrt{1+\eta}}{\sqrt{1-\delta_{M+|\mathcal{C}|}}-\frac{|\mathcal{C}|}{M}\sqrt{1+\delta_M}}\|\bm{\Psi h}\|_{2},\\
&\le\frac{2\sqrt{1+\eta}}{\sqrt{1-\delta_{M+|\mathcal{C}|}}-\frac{|\mathcal{C}|}{M}\sqrt{1+\delta_M}}\cdot\epsilon,
\end{align*}
which yields the proof with the remaining arguments from Theorem 1 of \cite{Candes06d}.
\end{proof}
%



In the case of recovery with no truncation error, that is $\epsilon = 0$, we expand on the consideration of the parameter $\beta$ in the above proof. We note that results for the case of $\epsilon = 0$ may not guarantee that a sparsest solution to $\mathcal{P}_{0,\epsilon}$ has been found, but may help to verify that as sparse as possible a solution to $\bm{u}_1=\bm{\Psi}\bm{c}_1$ has been found. Stated another way, the computed solution that recovers $\bm{u}_1$ may have verifiable sparsity, where $\bm{u}_1$ is close to $\bm{u}$.

We show how $\bm{W}$ and $\mathcal{C}$ affect the recovery when $\epsilon = 0$ through the null-space of $\bm{\Psi}$. Specifically, recall that the difference between any two solutions to $\bm{\Psi}\bm{c}=\bm{u}$ is a vector in the null-space of $\bm{\Psi}$, denoted by $\mathcal{N}(\bm{\Psi})$. It follows that
\begin{align}
\label{eqn:betadef}
\beta_{\bm{W}} = \mathop{\max}\limits_{\bm{c}\in\mathcal{N}(\bm{\Psi})}\frac{\|\bm{W}\bm{c}\|_{\mathcal{C},1}}{\|\bm{W}\bm{c}\|_{\mathcal{C}^c,1}},
\end{align}
is a bound on $\beta$ in (\ref{eqn:betta_def}) for the case that $\epsilon=0$.

When $\beta_{\bm{W}}$ is small we notice that adding to the sparse solution, $\bm{c}_0$, any vector $\bm{c}\in\mathcal{N}(\bm{\Psi})$ will induce a relatively small change in $\|\bm{W}(\bm{c}_0+\bm{c})\|_{\mathcal{C},1}$ while inducing a larger change in $\|\bm{W}(\bm{c}_0+\bm{c})\|_{\mathcal{C}^c}$. We see that we may decrease $\beta_{\bm{W}}$ if we make $w_j$ smaller for $j\in\mathcal{C}$, and larger for $j\in\mathcal{C}^c$, and this is consistent with our intuition regarding the identification of weights. As such, for small $\beta_{\bm{W}}$ we expect that $\|\bm{c}+\bm{c}_0\|_1>\|\bm{c}_0\|_1$ for all $\bm{c}\in\mathcal{N}(\bm{\Psi})$, and the following theorem shows that a critical value for $\beta_{\bm{W}}$ is $1$.
\begin{theorem}
If $\beta_{\bm{W}}<1$, then finding a solution to $\mathcal{P}^{(\bm{W})}_{1,0}$ is identical to finding a solution to $\mathcal{P}_{0,0}$. This result is sharp in that if $\beta_{\bm{W}}\ge 1$, a solution to $\mathcal{P}^{(\bm{W})}_{1,0}$, may not be identical to any solution of $\mathcal{P}_{0,0}$.
\end{theorem}
\begin{proof}
Closely related to $\beta_{\bm W}$, we define the quantity $\gamma_{\bm W}$ given by
\begin{align}
\label{eqn:gammadef}
\gamma_{\bm{W}} = \mathop{\max}\limits_{\bm{c}\in\mathcal{N}(\bm{\Psi})}\frac{\|\bm{W}\bm{c}\|_{\mathcal{C},1}}{\|\bm{W}\bm{c}\|_1},
\end{align}
where the two constants are related by
\begin{align*}
\beta_{\bm{W}}&= (\gamma_{\bm{W}}^{-1}-1)^{-1}.
\end{align*}

Recalling that $\bm{c}_0$ is supported on $\mathcal{C}$, we have that
\begin{align*}
\|\bm{W}(\bm{c}+\bm{c}_0)\|_1&=\|\bm{W}(\bm{c}+\bm{c}_0)\|_{\mathcal{C},1}+\|\bm{W}\bm{c}\|_{\mathcal{C}^c,1}.
\end{align*}
Applying the reverse triangle inequality to $\|\bm{W}(\bm{c}+\bm{c}_0)\|_{\mathcal{C},1}$, we have that
\begin{align*}
\|\bm{W}(\bm{c}+\bm{c}_0)\|_1&\ge \|\bm{W}\bm{c_0}\|_{\mathcal{C},1}-\|\bm{W}\bm{c}\|_{\mathcal{C},1} + \|\bm{W}\bm{c}\|_{\mathcal{C}^c,1}.
\end{align*}
By the definition of $\gamma_{\bm{W}}$ in (\ref{eqn:gammadef}) we have that
\begin{align*}
\|\bm{W}\bm{c}\|_{\mathcal{C},1}&\le \gamma_{\bm{W}}\|\bm{W}\bm{c}\|_{1},\\
\|\bm{W}\bm{c}\|_{\mathcal{C}^c,1}&= \|\bm{W}\bm{c}\|_1 - \|\bm{W}\bm{c}\|_{\mathcal{C},1},\\
&\ge (1-\gamma_{\bm{W}})\|\bm{W}\bm{c}\|_1.
\end{align*}
It follows that
\begin{align*}
\|\bm{W}(\bm{c}+\bm{c}_0)\|_1&\ge \|\bm{W}\bm{c_0}\|_{\mathcal{C},1}-\gamma_{\bm{W}}\|\bm{W}\bm{c}\|_{\mathcal{C},1} + (1-\gamma_{\bm{W}})\|\bm{W}\bm{c}\|_1,\\
&=\|\bm{W}\bm{c_0}\|_{\mathcal{C},1} + (1-2\gamma_{\bm{W}})\|\bm{W}\bm{c}\|_1,
\end{align*}
which implies that when $\gamma_{\bm{W}}<0.5$, or equivalently when $\beta_{\bm{W}}<1$,
\begin{align*}
\|\bm{W}(\bm{c}+\bm{c}_0)\|_1&>\|\bm{W}\bm{c_0}\|_{\mathcal{C},1}=\|\bm{W}\bm{c_0}\|_{1},
\end{align*}
and as such $\bm{c}_0$ solves $\mathcal{P}^{(\bm{W})}_{1,0}$.
To show sharpness, let $\bm{W}$ be the identity matrix. For $\alpha> 0$ define $\bm{\Psi}$ and $\bm{u}$ by
\begin{align*}
\bm{\Psi} &=\left(\begin{array}{ccc}
\alpha&    0& 1\\
0    &\alpha& 1
\end{array}
\right); & \bm{u} &=\left(\begin{array}{c}
\alpha\\
\alpha
\end{array}
\right).
\end{align*}
Note that the solution to $\mathcal{P}_{0,0}$ is always $(0~0~\alpha)^{T}$, and as such $\beta_{\bm{W}} = \alpha/2$. If $\beta_{\bm{W}} = 1$, corresponding to $\alpha = 2$, then $(0~0~2)^{T}$ or $(1~1~0)^{T}$ are both solutions to $\mathcal{P}^{(\bm{W})}_{1,0}$ . If $\beta_{\bm{W}}>1$, corresponding to $\alpha > 2$, the solution to $\mathcal{P}^{(\bm{W})}_{1,0}$ is $(1~1~0)^{T}$. 

As an aside, we note that if $\beta_{\bm{W}} < 1$, corresponding to $\alpha < 2$, the unique solution to $\mathcal{P}^{(\bm{W})}_{1,0}$ is $(0~0~\alpha)^{T}$ as guaranteed by the theorem.
\end{proof}
This result suggests $\beta_{\bm{W}}$ as a measure of quality of $\bm{W}$ with smaller $\beta_{\bm{W}}$ being preferable. The following bound is useful in relating the recovery via weighted $\ell_1$-minimization of a particular $\bm{c}_0$ to a uniform recovery in terms of the one implied by the RIC.
\begin{theorem}
\label{the:beta_w_bound}
Let 
\begin{align*}
c &:= \mathop{\min}\limits_{i\in \mathcal{C}}w_i/\mathop{\max}\limits_{i\in \mathcal{C}^c}w_i;\\
C &:= \mathop{\max}\limits_{i\in \mathcal{C}}w_i/\mathop{\min}\limits_{i\in \mathcal{C}^c}w_i.
\end{align*}
It follows that,
\begin{align}
\label{eqn:BetaBounds}
c\beta_{\bm{I}} \le \beta_{\bm{W}} =\mathop{\max}\limits_{\bm{c}\in\mathcal{N}(\bm{\Psi})}\frac{\|\bm{W}\bm{c}\|_{\mathcal{C},1}}{\|\bm{W}\bm{c}\|_{\mathcal{C}^c,1}} \le C\beta_{\bm{I}}.
\end{align}
Further, 
\begin{align}
\label{eqn:RICBound}
\beta_{\bm{I}}\le \frac{\sqrt{2}\delta_{2|\mathcal{C}|}}{1-\delta_{2|\mathcal{C}|}},
\end{align}
where $\delta$ is a RIC.
\end{theorem}
\begin{proof}
We first note that (\ref{eqn:BetaBounds}) follows from the definition of $\beta_{\bm{W}}$ in (\ref{eqn:betadef}).
To show (\ref{eqn:RICBound}), note that by Lemma 2.2 of \cite{Candes08c}, it follows that for any vector $\bm{x}$ in the null space of $\bm{\Psi}$, 
\begin{align*}
\|\bm{x}\|_{\mathcal{C},1}&\le \frac{\sqrt{2}\delta_{2|\mathcal{C}|}}{1-\delta_{2|\mathcal{C}|}}\|\bm{x}\|_{\mathcal{C}^c,1},
\end{align*}
which shows the bound.
\end{proof}

To complete our discussion on the theoretical analysis of weighted $\ell_1$-minimization, we require a sufficiently small RIC $\delta$ to bound $\beta_{\bm I}$ and $\beta_{\bm W}$ in Theorem \ref{the:beta_w_bound}, and hence $\beta$ in (\ref{eqn:betta_def}). For this, we report the result of \cite[Theorem 4.3]{Rauhut12} -- on general bounded orthonormal basis $\{\psi_j\}$ -- specialized to the case of multi-variate Legendre PC expansions.

\begin{corollary}Let $\{\psi_j\}_{1\le j\le P}$ be a Legendre PC basis in $d$ independent random variables $\bm \Xi =(\Xi_1,\dots,\Xi_d)$ uniformly distributed over $[-1,1]^d$ and with a total degree less than or equal to $q$. Let the matrix $\bm\Psi$ with entries $\bm\Psi(i,j)=\psi_j(\bm\xi^{(i)})$ correspond to realizations of $\{\psi_j\}$ at $\bm \xi^{(i)}$ sampled independently from the measure of $\bm\Xi$. If
\begin{equation}
\label{eqn_RIC_bound}
N\ge C 3^q\delta^{-2}s\log^3(s)\log(P),
\end{equation}
then the RIC, $\delta_s$, of $\frac{1}{\sqrt{N}}\bm\Psi$ satisfies $\delta_s\le\delta$ with probability larger than $1-P^{-\gamma\log^3(s)}$. Here, $C$ and $\gamma$ are constants independent of $N,q$, and $d$.
\end{corollary}
\begin{proof} The proof is a direct consequence of Theorem 4.3 in  \cite{Rauhut12} by observing that $\{\psi_j\}_{1\le j\le P}$ admits a uniform bound $\sup_j\Vert\psi_j\Vert_{\infty} = 3^{\frac{q}{2}}$, see, e.g. \cite{Doostan11a}. 
\end{proof}

\begin{remark}[Weighted $\ell_1$-minimization vs. $\ell_1$-minimization] While our theoretical analyses provide insight on the accuracy of the solution to the weighted $\ell_1$-minimization problem $\mathcal{P}^{(\bm{W})}_{1,\epsilon}$ relative to the solution to $\mathcal{P}_{0,\epsilon}$ or $\mathcal{P}_{0,0}$, they do not provide conclusive comparison between the accuracy of the solution to $\mathcal{P}^{(\bm{W})}_{1,\epsilon}$ and the standard $\ell_1$-minimization problem $\mathcal{P}_{1,\epsilon}$. However, for cases where the choice of $\bm W$ is such that the constant $C$ in (\ref{eqn:BetaBounds}) is sufficiently smaller than 1, more accurate solutions may be expected from $\mathcal{P}^{(\bm{W})}_{1,\epsilon}$ than $\mathcal{P}_{1,\epsilon}$.
\end{remark}

\subsection{Choosing $\epsilon$ via cross validation}
\label{sec:cross_validation}

The choice of $\epsilon>0$ for the optimization problems in (\ref{eq:CS}) or (\ref{eq:CS_w}) is critical. If $\epsilon$ is too small, then $\bm{c}$ will overfit the data and give unfounded confidence in $\bm{u}(\bm{\Xi})$; if $\epsilon$ is too large, then $\bm{c}$ will underfit the data and give unnecessary diffidence in $\bm{u}(\bm{\Xi})$.  In this work, following \cite{Doostan11a}, the selection of $\epsilon$ is determined by cross-validation; here we divide the available data into two sets, a reconstruction set of $N_r$ samples used to calculate $\bm{c}_r$, and a validation set of $N_v$ samples to test this approximation. For the reconstruction set we let $\bm{c}_r(\epsilon_r)$ denote the calculated solution to (\ref{eq:CS}) or (\ref{eq:CS_w}) as a function of $\epsilon_r$, and in this manner identify an optimal $\epsilon$ which is then corrected based on $N_r$ and $N_v$. This algorithm is summarized below where the subscript indicates which data set is used in calculating the quantity: $r$ for the reconstruction set; $v$ for the validation set.
\begin{algorithm}[H]
\label{alg:cv}
\caption{Algorithm for choosing $\epsilon$ using cross-validation.}
\begin{algorithmic}
\STATE Randomly divide the $N$ samples of $\bm\Xi,u(\bm\Xi)$ into two sets, a reconstruction set with $N_r$ samples and a validation set with $N_v$ samples.
\STATE Let $\epsilon^{*}=\mathop{\arg\min}_{\epsilon_r>0}\Vert \bm{\Psi}_v\bm{c}_r(\epsilon_r)-\bm{u}_v\Vert_2$.
\STATE Return $\epsilon=\sqrt{\frac{N}{N_r}}\epsilon^{*}$.
\end{algorithmic}
\end{algorithm}
We note that the optimal $\epsilon$ is dependent on the algorithm used to calculate $\bm{c}_r$ as well as the data input into that algorithm. In this paper we set $N_r = \lfloor \frac{4}{5} N\rfloor $ and $N_v = N - N_r$. %


%
\section{Numerical examples}
\label{sec:examples}
%

In this section, we empirically demonstrate the accuracy of the weighted $\ell_1$-minimization approach in estimating statistics of solutions to two differential equations with random inputs. 


%
%
\subsection{Case I: Elliptic equation with stochastic coefficient}
\label{subsec:L1_example_PDE}

We first consider the solution of an elliptic realization of (\ref{eqn:PDE_operator}) in one spatial dimension, defined by
\begin{eqnarray}
\label{eqn:PDE}
&&-\nabla\cdot \left(a(x,\bm{\Xi})\nabla u(x,\bm{\Xi})\right)=1 \quad x\in \mathcal{D}=(0,1),\nonumber \\
&&u(0,\bm{\Xi})=u(1,\bm{\Xi})=0.
\end{eqnarray}

We assume that the diffusion coefficient $a(x,\bm{\Xi})$ is modeled by the expansion
\begin{equation*}
a(x,\bm{\Xi})=\bar{a}(x)+\sigma_a\sum_{k=1}^{d}\sqrt{\lambda_k}\varphi_{k}(x)\Xi_{k},    
\end{equation*}
in which the random variables $\{\Xi_{k}\}_{k=1}^d$ are independent and uniformly distributed on $[-1,1]$. Additionally, $\{\varphi_k\}_{k=1}^d$ are the eigenfunctions of the Gaussian covariance kernel 
\begin{equation}
\label{eqn:kernel}
C_{aa}(x_1,x_2) = \exp{\left[-\frac{(x_1-x_2)^2}{l_c^2}\right]}, \nonumber
\end{equation}
corresponding to $d$ largest eigenvalues $\{\lambda_k\}_{k=1}^d$ of $C_{aa}(x_1,x_2)$ with correlation length $l_c=1/16$. In our numerical tests, we set $\bar{a}(x)=0.1$, $\sigma_a=0.021$, and $d=40$ resulting in strictly positive realizations of $a(x,\bm\Xi)$. Noting that $d$ represents the dimension of the problem in stochastic space, the Legendre PC basis functions for this problem are chosen as in (\ref{eqn:PCProdDef}), where we use an incomplete third order truncation, i.e., $q=3$, with only $P=2500$ basis functions. The PC basis functions $\{\psi_j\}$ are sorted such that, for any given order $q$, the random variables $\Xi_k$ with smaller indices $k$ appear first in the basis. The quantity of interest is $u(0.5,\bm\Xi)$, the solution in the middle of the spatial domain.

\subsubsection{Setting weights $w_j$}
\label{subsubsec:decay}

Recently, work has been done to derive estimates for the decay of the coefficients $c_{\bm\alpha}(\bm x)$ in the Legendre PC expansion of the solution $u(x,\bm{\Xi})\approx \sum_{\bm{\alpha}}c_{\bm\alpha}(\bm x)\psi_{\bm\alpha}(\bm\Xi) $ to problem (\ref{eqn:PDE}),~\cite{Bieri09c, Beck12a, Migliorati11}. Such estimates allow us to identify {\it a priori} knowledge of $\bm{c}$ and set the weights $w_j$ in the weighted $\ell_1$-minimization approach. In particular, following ~\cite[Proposition~3.1]{Beck12a}, the coefficients $c_{\bm\alpha}$ admit the bound 
\begin{equation}
\label{eqn:decay_rate}
\Vert c_{\bm\alpha}\Vert_{H^1_0(\mathcal{D})}\leq C_0 \frac{\vert\bm{\alpha}\vert !}{\bm{\alpha} !} e^{-\sum_{k=1}^d {g_k\alpha_k}},\quad g_k=-\log\left(r_k/(\sqrt{3}\log 2)\right),
\end{equation}
for some $C_0> 0$ and $\bm{\alpha}!=\prod_{k=1}^{d}{\alpha_k !}$. The coefficients $r_k$ in (\ref{eqn:decay_rate}) are given by $r_k=\frac{\sigma_a\sqrt{\lambda_k}\Vert\varphi_k\Vert_{L^{\infty}(D)}}{a_{\min}}$, where $a_{\min} = \bar{a} - \sigma_a\sum_{k=1}^d\sqrt{\lambda_k}\Vert\varphi_k\Vert_{L^{\infty}(D)}$. As suggested in \cite{Beck12a}, a tighter bound on $\Vert c_{\bm\alpha}\Vert_{H^1_0(\mathcal{D})}$ is obtained when the $g_k$ coefficients are computed numerically using one-dimensional analyses instead of the theoretical values given in (\ref{eqn:decay_rate}). Specifically, for each $k$, the random variables $\Xi_j$, $j\ne k$, in (\ref{eqn:PDE}) are set to their mean values and the PCE coefficients $c_{\alpha_k}$ of the corresponding solution -- now one-dimensional at the stochastic level --  are computed via, for instance, least-squares regression or sufficiently high level stochastic collocation. Notice that the total cost of such one-dimensional calculations depends linearly on $d$. Using these $c_{\alpha_k}$ values, the coefficient $g_k$ is computed from the one-dimensional version of (\ref{eqn:decay_rate}), i.e., $\vert c_{\alpha_k}\vert\sim e^{-g_k\alpha_k}$. In the present study, we adopt this numerical procedure to estimate each $g_k$.

As depicted in Fig.~\ref{fig:d40:soln_decay}, the bound in (\ref{eqn:decay_rate}) allows us to identify an anticipated $\bm{c}$, which we use for setting the weights $w_j$ in the weighted $\ell_1$-minimization approach. The magnitude of reference coefficients was calculated by the regression approach of \cite{Hosder06} using a sufficiently large number of solution realizations.
\begin{figure}[H]
\centering
\includegraphics[width=1\textwidth]{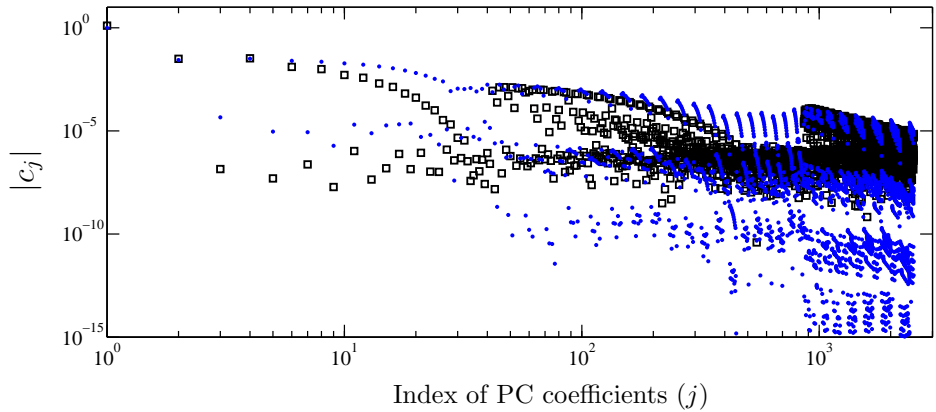}
\put(-235,0){\footnotesize Index of PC coefficients ($j$)}
\put(-380,80){\begin {sideways} $\vert c_j\vert$ \end{sideways}}
\caption{Polynomial chaos coefficients $\bm{c}$ of $u(0.5,\bm{\Xi})$ and the corresponding analytical bounds obtained from (\ref{eqn:decay_rate}) ({\scriptsize $\square$} reference; {\color{blue}$\bullet$} analytical bound).}
\label{fig:d40:soln_decay}
\end{figure}
We see that the reference values $\vert c_j\vert$ associated with some of the second and third degree basis functions decay slower than anticipated, but that the estimate is a reasonable guess without the use of realizations of $u(\bm x,\bm{\Xi})$.

\subsubsection{Results}
\label{subsubsec:result_d40}
%
%

To demonstrate the convergence of the standard and weighted $\ell_1$-minimization, we consider an increasing number $N=\{81,200,1000\}$ of random solution samples. For each analysis, we estimate the truncation error tolerance $\epsilon$ in (\ref{eq:CS}) based on the cross-validation algorithm described in Section \ref{sec:cross_validation}. To account for the dependency of the compressive sampling solution on the choice of realizations, for each $N$, we perform $100$ replications of standard and weighted $\ell_1$-minimization, corresponding to independent solution realizations. We then generate uncertainty bars on solution accuracies based on these replications. 

Fig.~\ref{fig:d40} displays a comparison between the accuracy of $\ell_1$-minimization, weighted $\ell_1$-minimization, iteratively re-weighted $\ell_1$-minimization, and (isotropic) sparse grid stochastic collocation with Clenshaw-Curtis abscissas. The level one sparse grid contains $N=81$ points. In particular, we observe that both $\ell_1$-minimization and weighted $\ell_1$-minimization result in smaller standard deviation and root mean square (rms) errors, compared to the stochastic collocation approach. Additionally, the weighted $\ell_1$-minimization using the analytical decay of $\vert c_{\bm\alpha}\vert$ outperforms the iteratively re-weighted $\ell_1$-minimization. Moreover, for small sample sizes $N$, the weighted $\ell_1$-minimization outperforms the non-weighted approach. This is expected as the prior knowledge on the decay of $\vert c_{\bm\alpha}\vert$ has comparable effect on the accuracy as the solution realizations do. In fact, the trade-off between the prior knowledge (in the form of weights $w_j$) and the solution realizations (data) may be best seen in a Bayesian formulation of the compressive sampling problem (\ref{eq:CS}). We refer the interested reader to \cite{Tipping01,Ji08} for further information on this subject.

In the presence of the {\it a priori} estimates of the PC coefficients, one may consider solving a weighted least-squares regression problem $\mathcal{P}_{2,\epsilon}^{(\bm W)}\equiv\{\mathop{\arg\min}\limits_{\bm{c}_{\mathcal{C}}}\Vert \bm W\bm{c}_{\mathcal{C}}\Vert_2 : \Vert\bm{\Psi}_{\mathcal{C}}\bm{c}_{\mathcal{C}}-\bm{u}\Vert_2\leqslant\epsilon\}$, in which $\bm{c}_{\mathcal{C}}\in\mathbb{R}^{P}$ denotes vectors supported on a set $\mathcal{C}$ with cardinality $\vert \mathcal{C}\vert\le N$ identified based on the decay of PC coefficients. For example, to generate a well-posed weighted least-squares problem, $\mathcal{C}$ may contain the indices associated with $\vert \mathcal{C}\vert\le \lfloor N/2\rfloor$ largest (in magnitude) PC coefficients from (\ref{eqn:decay_rate}). Stated differently, the estimates of PC coefficients may be utilized to form least-squares problems for {\it small} subsets of the PC basis function that are expected to be important. However, our numerical experiments indicate that, unlike in the case of weighted $\ell_1$-minimization, the accuracy of such an approach is sensitive to the quality of the PC coefficient estimates, based on which $\mathcal{C}$ is set. Fig. \ref{fig:only_weight_d40} presents an illustration of such observation.

\begin{figure}[H]
  \centering
  \subfloat[Relative error in mean]{\label{fig:d40:mean}\includegraphics[width=0.5\textwidth]{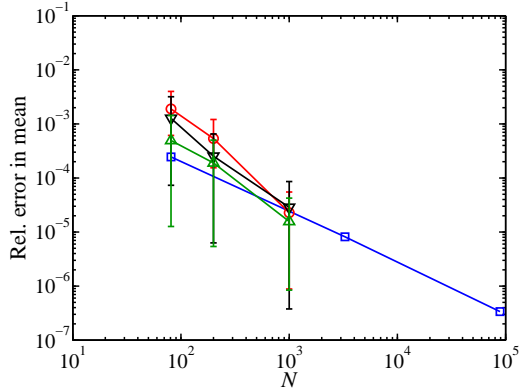}}
  
  \subfloat[Relative error in standard deviation]{\label{fig:d40:sd}\includegraphics[width=0.5\textwidth]{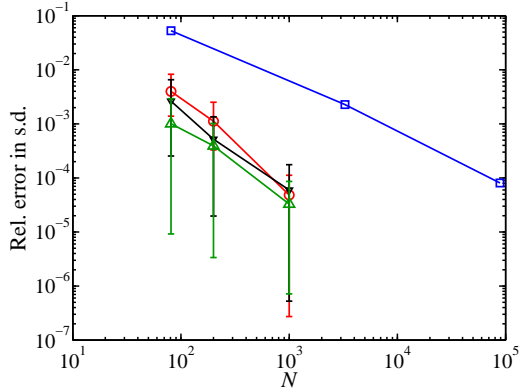}}
  
  \subfloat[Relative rms error]{\label{fig:d40:mse}\includegraphics[width=0.5\textwidth]{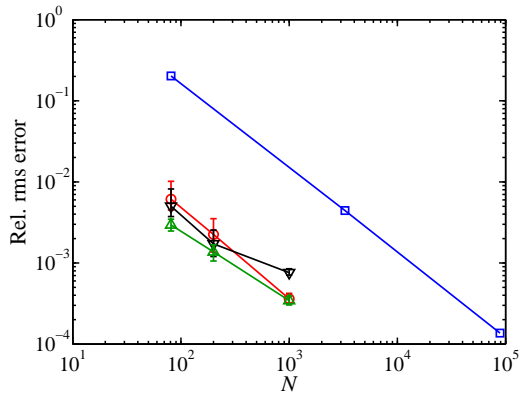}}

  \caption{Comparison of relative error in statistics of $u(0.5,\bm{\Xi})$ for $\ell_1$-minimization, weighted $\ell_1$-minimization, and isotropic sparse grid stochastic collocation (with Clenshaw-Curtis abscissas) for the case of the elliptic equation. The uncertainty bars are generated using 100 independent replications for each samples size $N$ (\usebox{\LegendeCL}~$\ell_1$-minimization;~\usebox{\LegendeW}~weighted $\ell_1$-minimization;~\usebox{\LegendeIW}~iteratively re-weighted $\ell_1$-minimization;~\usebox{\LegendeSC}~stochastic collocation).}
  \label{fig:d40}
\end{figure}

\begin{figure}[H]
  \centering
  \includegraphics[width=0.7\textwidth]{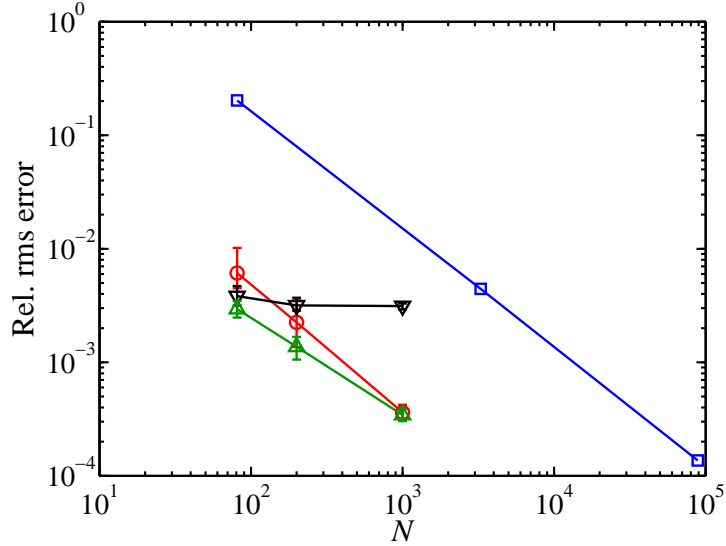}
  \caption{Comparison of relative rms error for $\ell_1$-minimization, weighted $\ell_1$-minimization, weighted least-squares regression, and sparse grid collocation for the case of the elliptic equation. In the weighted least-squares approach the set $\mathcal{C}$ with cardinality $\vert \mathcal{C}\vert = \lfloor N/2\rfloor$ contains the indices of the largest (in magnitude) upper bounds on the PC coefficients (\usebox{\LegendeCL}~$\ell_1$-minimization;~\usebox{\LegendeW}~weighted $\ell_1$-minimization;~\usebox{\LegendeIW}~weighted least-squares regression;~\usebox{\LegendeSC}~stochastic collocation).}
  \label{fig:only_weight_d40}
\end{figure}

\subsection{Case II: Thermally driven flow with stochastic boundary temperature}
\label{sec:L1_example_cavity}

Following \cite{LeMaitre02b,LeMaitre10,LeQuere91}, we next consider a $2$-D heat driven square cavity flow problem, shown in Fig.~\ref{subfig:cavity}, as another realization of (\ref{eqn:PDE_operator}). The left vertical wall has a deterministic, constant temperature $\tilde{T}_h$, referred to as the hot wall, while the right vertical wall has a stochastic temperature $\tilde T_c<\tilde T_h$ with constant mean $\bar{\tilde T}_c$, referred to as the cold wall. Both top and bottom walls are assumed to be adiabatic. The reference temperature and the reference temperature difference are defined as $\tilde T_{ref}=(\tilde{T}_h+\bar{\tilde T}_c)/2$ and $\Delta \tilde T_{ref}=\tilde{T}_h-\bar{\tilde T}_c$, respectively. In dimensionless variables, the governing equations (in the small temperature difference regime, i.e., Boussinesq approximation) are given by
\begin{equation}
\begin{aligned}
&\frac{\partial \bm{u}}{\partial t} + \bm{u}\cdot\nabla\bm{u}=-\nabla p + \frac{\text{Pr}}{\sqrt{\text{Ra}}}\nabla^2\bm{u}+\text{Pr}T\bm{\hat{y}},\\
& \nabla\cdot\bm{u}=0,\\ \label{eqn:cavity}
&\frac{\partial T}{\partial t}+ \nabla\cdot(\bm{u}T)=\frac{1}{\sqrt{\text{Ra}}}\nabla^2T,
\end{aligned}
\end{equation}
where $\bm{\hat{y}}$ is the unit vector $(0,1)$, $\bm{u}=(u,v)$ is velocity vector field, $T=(\tilde{T}-\tilde T_{ref})/\Delta \tilde T_{ref}$ is normalized temperature ($\tilde{T}$ denotes non-dimensional temperature), $p$ is pressure, and $t$ is time. Non-dimensional Prandtl and Rayleigh numbers are defined, respectively, as $\text{Pr}={\tilde \mu}\tilde c_p/\tilde \kappa$ and $\text{Ra}={\tilde \rho}{g}\beta\Delta \tilde T_{ref}{\tilde L}^3/({\tilde \mu}{\tilde \kappa})$, where the superscript tilde ($\tilde{~}$) denotes the non-dimensional quantities. Specifically, ${\tilde \rho}$ is density, ${\tilde L}$ is reference length, ${g}$ is gravitational acceleration, ${\tilde \mu}$ is molecular viscosity, ${\tilde \kappa}$ is thermal diffusivity, and the coefficient of thermal expansion is given by $\beta$. In this example, the Prandtl and Rayleigh numbers are set to $\text{Pr}=0.71$ and $\text{Ra}=10^6$, respectively. For more details on the non-dimensional variables in (\ref{eqn:cavity}), we refer the interested reader to \cite{LeQuere91,LeMaitre02b,LeMaitre10}.

On the cold wall, we apply a (normalized) temperature distribution with stochastic fluctuations of the form
\begin{equation}
\begin{aligned}
T_c(x=1,y,\bm\Xi)=\bar{T}_c + T_c',\\
T_c'=\sigma_T\sum_{i=1}^{d}\sqrt{\lambda_i}\varphi_i(y)\Xi_i,
\end{aligned}
\label{eqn:coldwall}
\end{equation}
where $\bar{T}_c$ is a constant mean temperature. In (\ref{eqn:coldwall}), $\Xi_i$, $i=1,\dots,d$, are independent random variables uniformly distributed on $[-1,1]$. $\{\lambda_i\}_{i=1}^d$ and $\{\varphi_i(y)\}_{i=1}^d$ are the $d$ largest eigenvalues and the corresponding eigenfunctions of the exponential covariance kernel
\begin{equation*}
C_{T_cT_c}(y_1,y_2) = \exp{\left(-\frac{\vert y_1-y_2\vert}{l_c}\right)},
\end{equation*}
where $l_c$ is the correlation length. Following ~\cite{Ghanem03}, the eigenpairs $(\lambda_i,\varphi_i(y))$ in~(\ref{eqn:coldwall}) are, respectively, given by
\begin{equation*}
\label{eqn:eigenvalue}
\lambda_i=\frac{2l_c}{l_c^2\omega_i^2+1},
\end{equation*}
and
\begin{equation*}
\label{eqn:eigenvec}
\varphi_i(y)=\left\{
\begin{aligned}
&\frac{\cos(\omega_i y)}{\sqrt{0.5+\frac{\sin(\omega_i)}{2\omega_i}}},~~& i~~\text{is odd},\\
&\frac{\sin(\omega_i y)}{\sqrt{0.5-\frac{\sin(\omega_i)}{2\omega_i}}},~~& i~~\text{is even},\\
\end{aligned}\right.
\end{equation*}
where each $\omega_i$ is a root of
\begin{equation*}
\label{eqn:omega}
\omega_i+(1/l_c)\tan(0.5\omega_i) = 0.
\end{equation*}

\begin{figure}[H]
\centering
\subfloat[Schematic of the geometry and boundary conditions.]{\label{subfig:cavity}
\includegraphics[height=0.27\textheight]{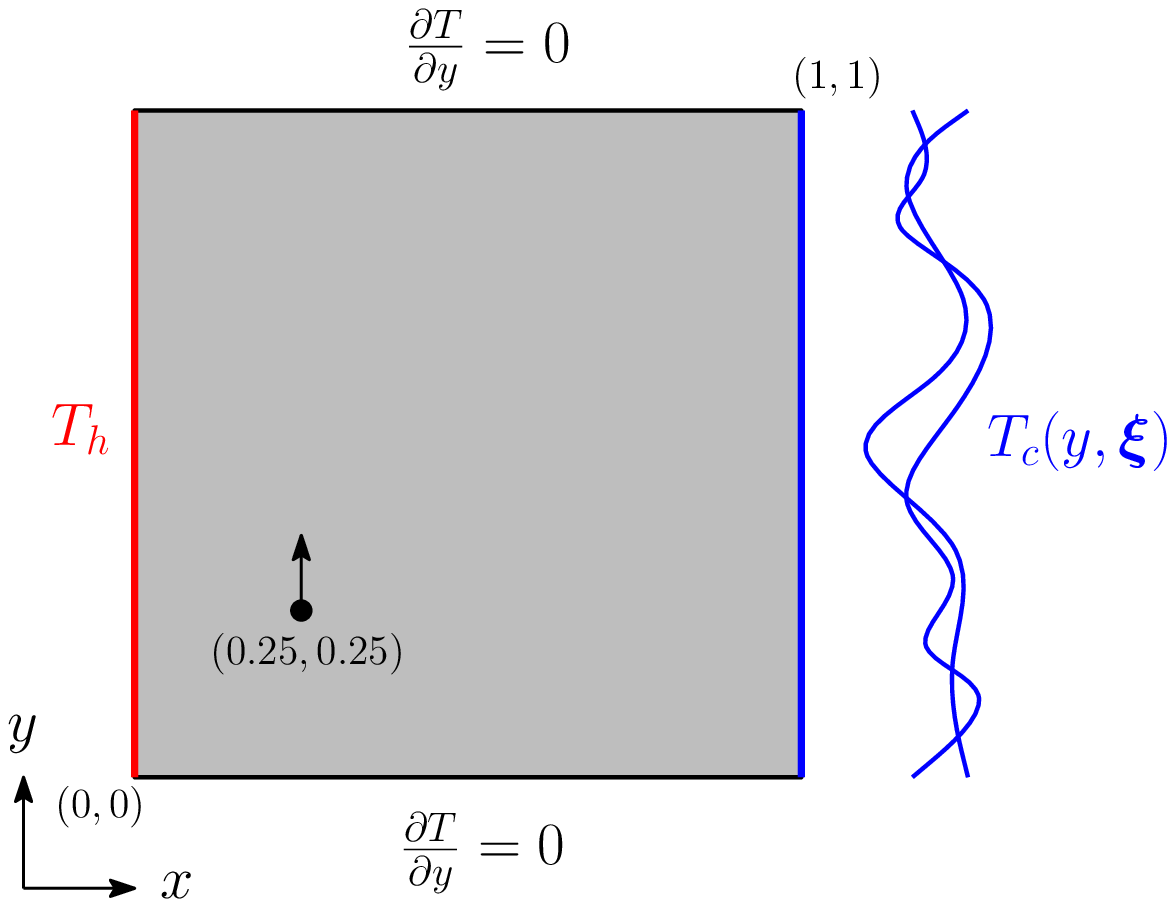}
}
\subfloat[A realization of $T_c(x=1,y)$.]{
\includegraphics[height=.25\textheight]{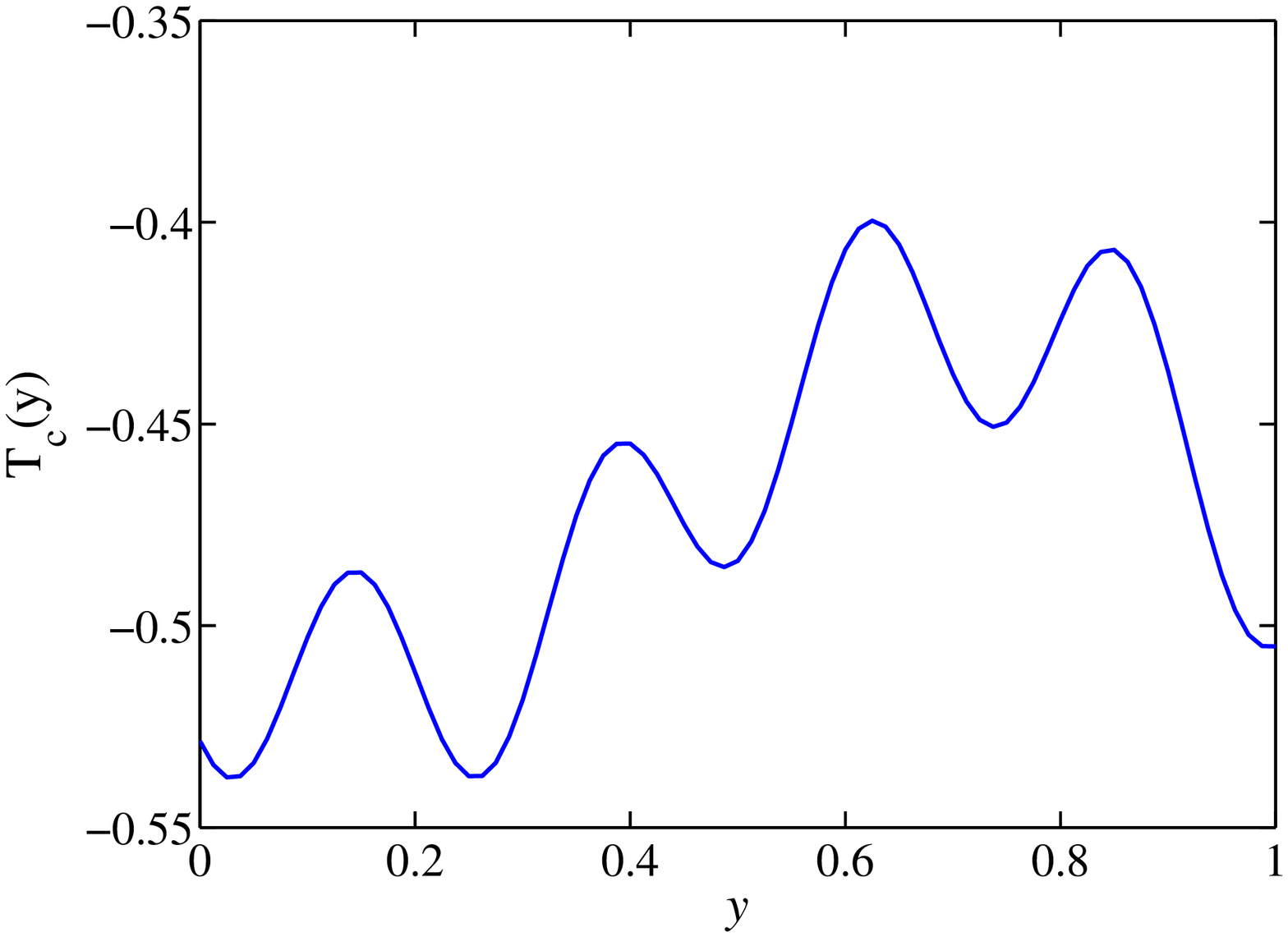}
\label{subfig:Tc}
}
\caption{Illustration of the cavity flow problem.}
\end{figure}

In our numerical test we let $(T_h,\bar T_c)=(0.5,-0.5)$, $d=20$, $l_c=1/21$, and $\sigma_T=11/100$. A realization of the cold wall temperature $T_c$ is shown in Fig.~\ref{subfig:Tc}. 
Our quantity of interest, the vertical velocity component at $(x,y)=(0.25,0.25)$ denoted by $v(0.25,0.25)$, is expanded in the Legendre PC basis of total degree $q=4$ with only the first $P=2500$ basis functions retained, as described in the case of the elliptic problem. We seek to accurately reconstruct $v(0.25,0.25)$ with $N< P$ random samples of $\bm{\Xi}$ and the corresponding realizations of $v(0.25,0.25)$. 

\subsubsection{Approximate bound on PC coefficients}
\label{subsec:rates}

In order to generate the weights $w_j$ for the weighted $\ell_1$-minimization reconstruction of $v(0.25,0.25)$, we derive an approximate bound on the PC coefficients of the velocity $v$ in (\ref{eqn:cavity}) at a fixed point in space.

For the interest of notation, we start by rewriting $T_c^{\prime}$ in (\ref{eqn:coldwall}) as
\begin{equation}
\label{eqn:T'_re}
T_c^{\prime}(y,\bm\Xi)=\sum_{i=1}^{d}\nu_i(y)\Xi_i,
\end{equation}
where $\nu_i(y)$, $i=1,\dots,d$, is given by
\begin{equation*}
\label{eqn:A_i}
\nu_i(y)=\sigma_T\sqrt{\frac{\lambda_i}{0.5+(-1)^{i-1}{\sin(\omega_i)}/{2\omega_i}}}\sin\left(\omega_i y+\frac{\pi}{2}\left((-1)^i+1\right)\right).
\end{equation*}

We write the PC expansion of $v$ as $v=\sum_{j}c_{j}\psi_{j}(\bm{\Xi})$ and seek approximate bounds on $\vert c_{j}\vert$ to set the weights $w_j$ in the weighted $\ell_1$-minimization results. By the orthonormality of the PC basis, $c_{j}$ is
\begin{equation}
\label{eq:exact_v_coef}
c_{j}=\int_{[-1,1]^d} v(\bm\xi)\psi_{j}(\bm{\xi})\left(\frac{1}{2}\right)^d d\bm{\xi}.
\end{equation}
To approximately bound the coefficients $c_{j}$, we examine the functional Taylor series expansion of $v$ around $v=v(\bar{T}_c)$. Note that by an appropriate definition of functional derivatives $\frac{\delta^k v}{\delta T^k_c}$ of $v$ with respect to $T_c$, see, e.g., \cite{Volterra59},
\begin{equation}
\label{eqn:Taylor_main}
v(\bm\Xi)=\mathop{\sum}\limits_{k=0}^{\infty}\frac{1}{k!}\int_{[0,1]^k}\frac{\delta^kv}{\delta \bar T^k_c}(\bm{y},\bm{\Xi})\mathop{\prod}\limits_{j=1}^kT^{\prime}_c(y_j,\bm\Xi)d\bm{y},
\end{equation}
where $y_j$ is a copy of the spatial coordinate variable $y$. Plugging (\ref{eqn:Taylor_main}) in (\ref{eq:exact_v_coef}), we arrive at
\begin{equation}
\label{eq:decay_cavity_bound}
c_{j} = \int_{[-1,1]^d}\psi_{j}(\bm{\xi}) \mathop{\sum}\limits_{k=0}^{\infty}\frac{1}{k!}\int_{[0,1]^k}\frac{\delta^kv}{\delta \bar T^k_c}(\bm{y},\bm{\xi})\mathop{\prod}\limits_{j=1}^kT^{\prime}_c(y_j,\bm{\xi})\left(\frac{1}{2}\right)^d d\bm{y}d\bm{\xi}.
\end{equation}

To handle the functional derivatives, we consider the dimensional relation
\begin{equation}
\label{eqn:scaling}
\bigg\vert\frac{\delta^kv}{\delta \bar T^k_c}(\bm{y})\bigg\vert\approx C\bigg\vert \frac{v(\bar{T}_c)}{\left(\bar{T}_c\right)^k}\bigg\vert,
\end{equation}
which we assume to hold uniformly in $\bm{y}$ and $\bm{\Xi}$, for some constant $C\ge 0$. This, together with (\ref{eqn:T'_re}), allows us to derive the approximate bound
\begin{equation}
\label{eqn:c_alpha_upper}
\vert c_{j}\vert\lessapprox C \vert v(\bar{T}_c)\vert \mathop{\sum}\limits_{k=0}^{\infty}\frac{1}{k! \vert\bar{T}_c \vert^k} \left\vert \int_{[-1,1]^d}\psi_{j}(\bm\xi)\left(\sum_{i=1}^d t_i\xi_i\right)^k \left(\frac{1}{2}\right)^d d\bm{\xi}\right\vert,
\end{equation}
where $t_i = \int_{0}^1\nu_i(y)dy$. In (\ref{eqn:c_alpha_upper}), the approximation comes from the assumption (\ref{eqn:scaling}) on the functional derivatives. To evaluate the RHS of (\ref{eqn:c_alpha_upper}), we consider a finite truncation of the sum and a Monte Carlo (or quadrature) estimation of the integral.  

In Fig.~\ref{fig:decay_rates}, we display the approximate upper bound on $\vert c_{j}\vert$ of $v(0.25,0.25)$ obtained from (\ref{eq:decay_cavity_bound}) by limiting $k$ to $4$. To generate a reference solution, we employ the least-squares regression approach of \cite{Hosder06} with $N=40,000$ random realizations of $v(0.25,0.25)$. For the accuracies of interest in this study, the convergence of this reference solution was verified. For the sake of illustration, we normalize the estimated $\vert c_{j}\vert$ so that $\vert c_{\bm 0}\vert$, the module of the approximate zero degree coefficient, matches its reference counterpart. Despite the rather strong assumption (\ref{eqn:scaling}) on the functional derivatives, we note that the resulting estimates of $\vert c_{j}\vert$ describe the trend of the reference values qualitatively well. As we shall see in what follows, such qualitative agreement is sufficient for the weighted $\ell_1$-minimization to improve the accuracy of the standard $\ell_1$-minimization for small samples sizes $N$. 

\begin{remark} We stress that the assumption (\ref{eqn:scaling}), while here lead to appropriate estimates of $\vert c_{j}\vert$ for our particular example of interest, it may not give equally reasonable estimates for other problems or choices of flow parameters, e.g., larger {$\mathrm Ra$} numbers. A weaker assumption on the functional derivatives in (\ref{eqn:scaling}), however, requires further study and is the subject of our future work.

\end{remark}

\begin{figure}[H]
\centering
\includegraphics[width=1\textwidth]{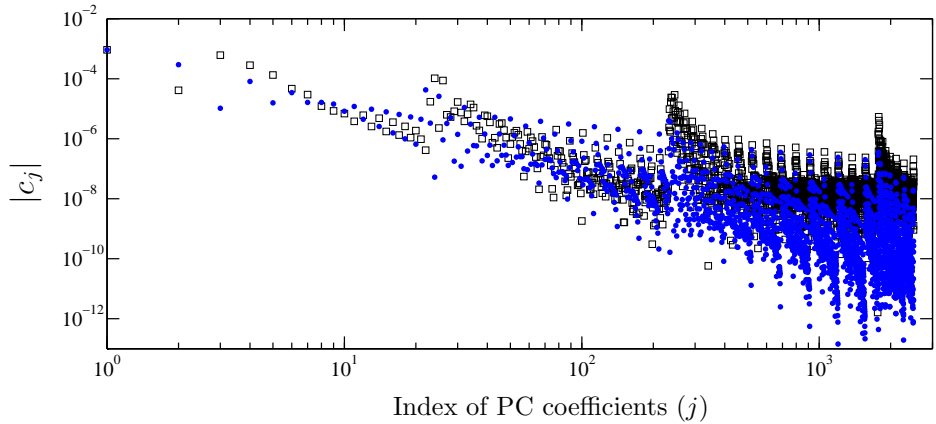}
\put(-235,0){\footnotesize Index of PC coefficients ($j$)}
\put(-380,80){\begin {sideways} $\vert c_j\vert$ \end{sideways}}
\caption{Approximate PC coefficients of $v(0.25,0.25)$ vs. the reference coefficients obtained by least-squares regression using sufficiently large number of realizations of $v(0.25,0.25)$ ({\scriptsize $\square$} reference; {\color{blue}$\bullet$} approximate bound).}
\label{fig:decay_rates}
\end{figure}

\subsubsection{Results}
\label{subsec:results_cavity}

We provide results demonstrating the convergence of the statistics of $v(0.25,0.25)$ as a function of the number of realizations $N$. For this, we consider sample sizes $N=\{41,200,1000\}$ with $N=41$ corresponding to the number of grid points in level one sparse gird collocation using Clenshaw-Curtis abscissas. 

Fig.~\ref{fig:comparison} displays comparisons between the accuracies obtained to approximate $v(0.25,0.25)$. Similar to the previous example, the weighted $\ell_1$-minimization approach achieves superior accuracy, particularly for the small sample size $N=41$. The results obtained for the iteratively re-wighted $\ell_1$-minimization correspond to $\epsilon_w=5\times 10^{-2}\cdot \hat{c}_1$, where $\hat{c}_1$ is the sample average of  $v(0.25,0.25)$. This leads to the smallest average rms errors among the trial values  $\epsilon_w = \{5\times 10^{-2},5\cdot 10^{-3},5\times 10^{-4}\}\cdot\hat{c}_1$. To show the sensitivity of this approach to the choice of $\epsilon_w$, we present rms error plots in Fig.~\ref{fig:sensitivity_cavity} corresponding to multiple values of $\epsilon_w$. In particular, for the cases of  $\epsilon_w = \{5\times 10^{-3},5\times 10^{-4}\}\cdot\hat{c}_1$, when $N=1000$ we observe loss of accuracy compared to the standard $\ell_1$-minimization. On the other hand, the weighted $\ell_1$-minimization results are relatively insensitive to the choice of $\epsilon_w$, and best performance is obtained with $\epsilon_w=5\times 10^{-4}\cdot\hat{c}_1$, i.e., the smallest and most intuitive value among the trials. 

We note that the rather poor performance of the sparse grid collocation is due to the relatively large contributions of some of the higher order PC modes, as may be observed from Fig.~\ref{fig:decay_rates}. Fig.~\ref{fig:approximation_200} shows the magnitude of PC coefficients of $v(0.25,0.25)$ obtained using standard and weighted $\ell_1$-minimization with $N=\{200,1000\}$ samples. The better approximation quality of the weighted $\ell_1$-minimization may be seen particularly from Figs.~\ref{fig:coefficients:200} and \ref{fig:coefficients:200_w}. Finally, in Fig.~\ref{fig:only_weight_cavity}, we present a comparison between the rms errors obtained from $\ell_1$-minimization, weighted $\ell_1$-minimization, weighted least-squares regression, and sparse grid stochastic collocation. The weighted least-squares regression approach performs poorly for $N=\{200,1000\}$ as some of the basis functions are selected incorrectly given the approximate bounds on the PC coefficients.   

%
\section{Conclusion}
\label{sec:conclusion}
%

Within the context of compressive sampling of sparse polynomial chaos (PC) expansions, we introduced a {\it weighted $\ell_1$-minimization} approach, wherein we utilized {\it a priori} knowledge on PC coefficients to enhance the accuracy of the standard $\ell_1$-minimization. The {\it a priori} knowledge of PC coefficients may be available in the form of analytical decay of the PC coefficients, e.g., for a class of linear elliptic PDEs with random data, or derived from simple dimensional analysis. These {\it a priori} estimates, when available, can be used to establish weighted $\ell_1$ norms that will further penalize small PC coefficients, and consequently improve the sparse approximation. We provided analytical results guaranteeing the convergence of the weighted $\ell_1$-minimization approach. 

The performance of the proposed weighted $\ell_1$-minimization approach was demonstrated through its application to two test cases. For the first example, dealing with a linear elliptic equation with random coefficient, existing analytical bounds on the magnitude of PC coefficients were adopted to establish the weights. In the second case, for a thermally driven flow problem with stochastic temperature boundary condition, we derived an approximate bound for the PC coefficients via a functional Taylor series expansion and a simple dimensional analysis. In both cases we demonstrated that the weighted $\ell_1$-minimization approach outperforms the non-weighted counterpart. Furthermore, better accuracies were obtained using the weighted $\ell_1$-minimization approach as compared to the iteratively re-weighted $\ell_1$-minimization. Numerical experiments illustrate the sensitivity of the latter approach, unlike the former, with respect to the choice of a parameter defining the weights. Finally, we demonstrated that selection of subsets of PC basis and solving well-posed weighted least-squares regression may result in poor accuracies. 

While our numerical and analytical results were for the case of Legendre PC expansions, our work may be extended to other choices of PC basis, such as those based on Hermite or Jacobi polynomials.

\section{Acknowledgements}
\label{sec:acknow}

We would like to thank Prof. Raul Tempone for bringing to our attention the use of the analytical PC estimates of the elliptic problem within the context of weighted $\ell_1$-minimization. We gratefully acknowledge the financial support of the Department of Energy under Advanced Scientific Computing Research Early Career Research Award DE-SC0006402. 

This work utilized the Janus supercomputer, which is supported by the National Science Foundation (award number CNS-0821794) and the University of Colorado Boulder. The Janus supercomputer is a joint effort of the University of Colorado Boulder, the University of Colorado Denver and the National Center for Atmospheric Research.

\begin{figure}[H]
  \centering
  \subfloat[Relative error in mean]{\label{fig:cavity:mean}\includegraphics[width=0.5\textwidth]{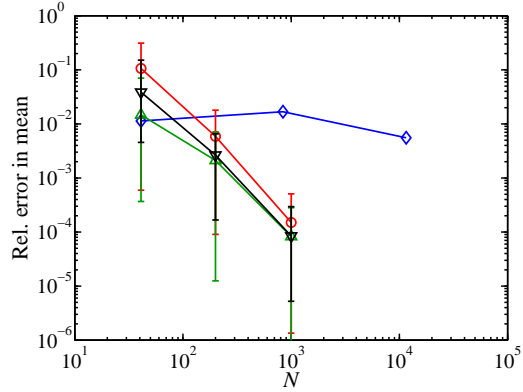}}
  
  \subfloat[Relative error in second moment]{\label{fig:cavity:sd}\includegraphics[width=0.5\textwidth]{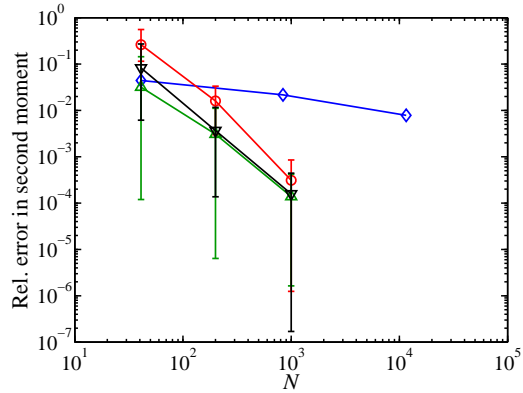}}
  
  \subfloat[Relative rms error]{\label{fig:cavity:mse}\includegraphics[width=0.5\textwidth]{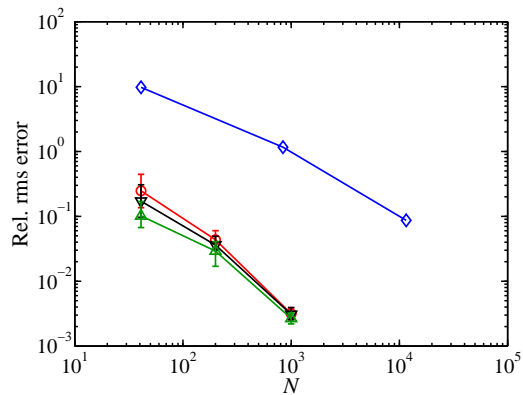}}

  \caption{Comparison of relative error in statistics of {\color{black}$v(0.25,0.25)$} computed via $\ell_1$-minimization, weighted $\ell_1$-minimization, iteratively reweighed $\ell_1$-minimization, and stochastic collocation. The error bars are generated using 100 independent replications with fixed samples size $N$ (\usebox{\LegendeCL}~$\ell_1$-minimization;~\usebox{\LegendeW}~weighted $\ell_1$-minimization;~\usebox{\LegendeIW}~iteratively re-weighted $\ell_1$-minimization;~\usebox{\LegendeSC}~stochastic collocation).}
  \label{fig:comparison}
\end{figure}

%

\begin{figure}[H]
  \centering
  \includegraphics[width=0.7\textwidth]{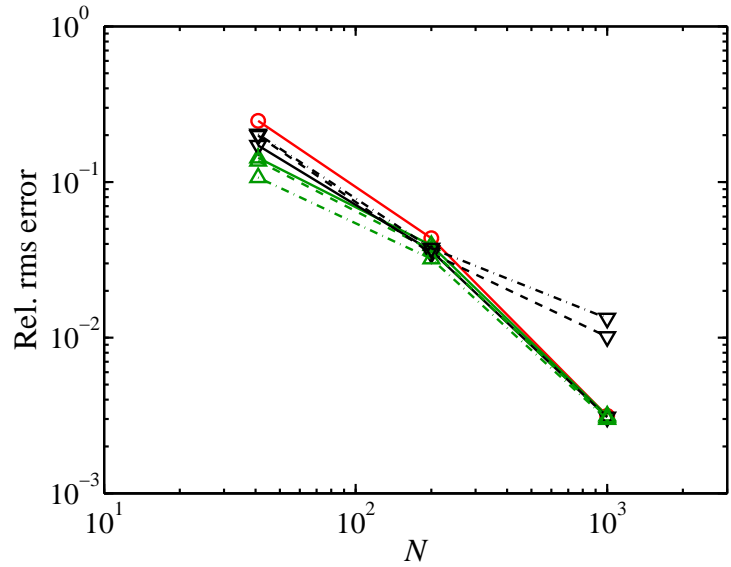}
  \caption{Relative average rms errors corresponding to multiple values of $\epsilon_w$ to set the weights $w_j$. The results demonstrate the sensitivity of the iteratively re-weighted approach to the choice of $\epsilon_p$ (\usebox{\LegendeCL}~$\ell_1$-minimization;~\usebox{\LegendeW}~weighted $\ell_1$-minimization;~\usebox{\LegendeIW}~iteratively re-weighted $\ell_1$-minimization; solid lines $\epsilon_w = 5\times 10^{-2}\cdot\hat{c}_1$; dashed lines $\epsilon_w = 5\times 10^{-3}\cdot\hat{c}_1$; dotted dashed lines $\epsilon_w = 5\times 10^{-4}\cdot\hat{c}_1$). Here, $\hat{c}_1$ is the sample average of $v(0.25,0.25)$.}
  \label{fig:sensitivity_cavity}
\end{figure}

\begin{figure}[H]
\centering
  \subfloat[$\ell_1$-minimization]{\label{fig:coefficients:200}\includegraphics[width=0.7\textwidth]{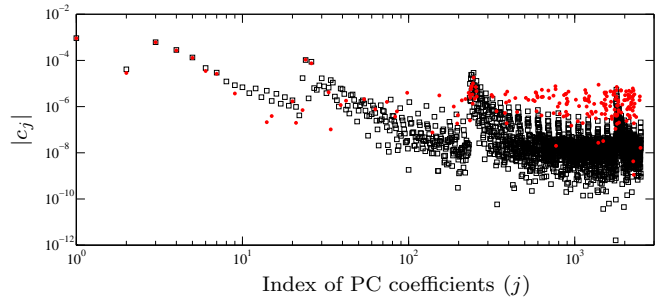}}
\put(-170,0){\scriptsize Index of PC coefficients ($j$)}
\put(-265,55){\begin {sideways}{\scriptsize $\vert c_j\vert$} \end{sideways}}

  \subfloat[Weighted $\ell_1$-minimization]{\label{fig:coefficients:200_w}\includegraphics[width=0.7\textwidth]{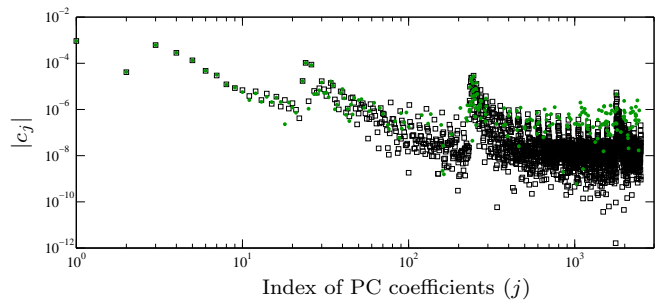}}  
\put(-170,0){\scriptsize Index of PC coefficients ($j$)}
\put(-265,55){\begin {sideways}{\scriptsize $\vert c_j\vert$} \end{sideways}}

  \subfloat[$\ell_1$-minimization]{\label{fig:coefficients:1000}\includegraphics[width=0.7\textwidth]{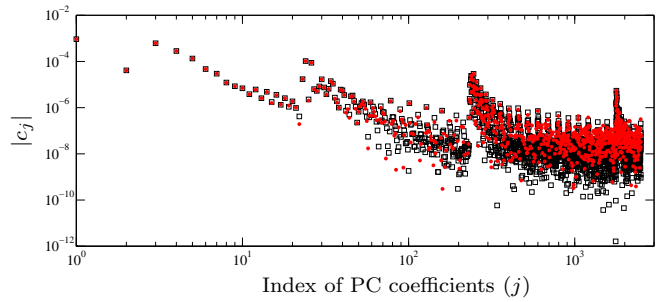}}
\put(-170,0){\scriptsize Index of PC coefficients ($j$)}
\put(-265,55){\begin {sideways}{\scriptsize $\vert c_j\vert$} \end{sideways}}
  
  \subfloat[Weighted $\ell_1$-minimization]{\label{fig:coefficients:1000_w}\includegraphics[width=0.7\textwidth]{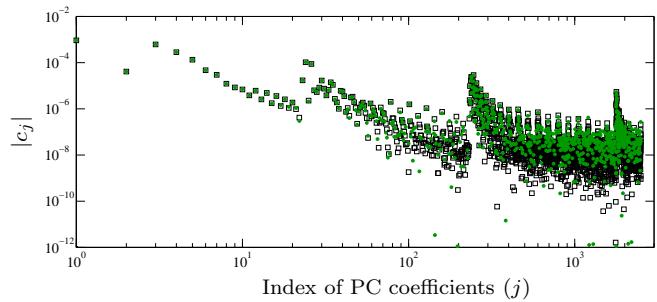}}
\put(-170,0){\scriptsize Index of PC coefficients ($j$)}
\put(-265,55){\begin {sideways}{\scriptsize $\vert c_j\vert$} \end{sideways}}
  
\caption{Approximation of PC coefficients of $v(0.25,0.25)$ using $N=200$ samples (a), (b) and $N=1000$ samples (c), (d) ({\scriptsize $\square$} reference; {\color{red}$\bullet$} $\ell_1$-minimization;  {\color{darkgreen}$\bullet$} weighted $\ell_1$-minimization).}
\label{fig:approximation_200}
\end{figure}

\begin{figure}[H]
  \centering
  \includegraphics[width=0.7\textwidth]{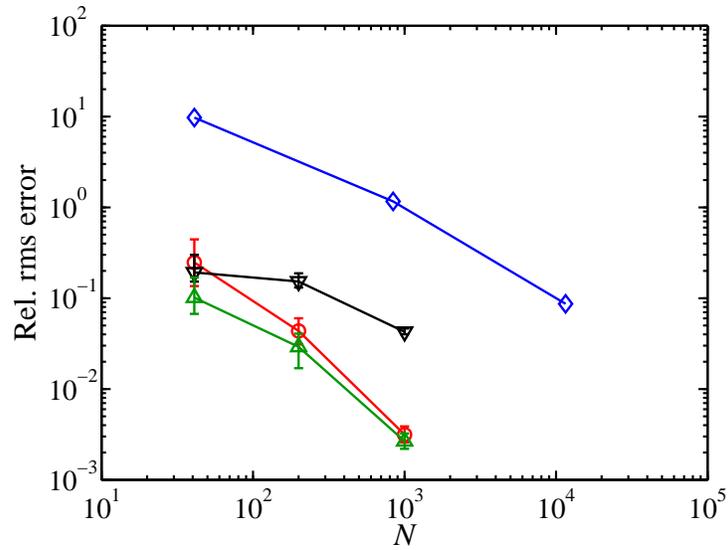}
  \caption{Comparison of relative rms error for $\ell_1$-minimization, weighted $\ell_1$-minimization, weighted least-squares regression, and sparse grid collocation for the cavity flow problem. In the weighted least-squares approach, the set $\mathcal{C}$ with cardinality $\vert \mathcal{C}\vert = \lfloor N/2\rfloor$ contains the indices of the largest (in magnitude) approximate upper bounds on the PC coefficients (\usebox{\LegendeCL}~$\ell_1$-minimization;~\usebox{\LegendeW}~weighted $\ell_1$-minimization;~\usebox{\LegendeIW}~weighted least-squares regression;~\usebox{\LegendeSC}~stochastic collocation).}
  \label{fig:only_weight_cavity}
\end{figure}

\newpage
\bibliographystyle{elsarticle-num}
\bibliography{AD_bib_v1}
\end{document}